\documentclass[11pt]{article}

\usepackage{amssymb,amsmath,amsthm}
\usepackage{enumerate}
\usepackage{cases}

\newcommand{\n}{\noindent}
\newcommand{\ovl}{\overline}
\newcommand{\intl}{\int\limits}

\newcommand{\bb}[1]{\mathbb{#1}}

\newcommand{\vp}{\varepsilon}
\newcommand{\sst}{\scriptstyle}

\newcommand{\ms}{\medskip}

\theoremstyle{plain}
\newtheorem{thm}{Theorem}[section]
\newtheorem{lem}{Lemma}[section]
\newtheorem{pro}{Proposition}[section]

\newtheorem{rem}{Remark}[section]

\theoremstyle{definition}

\numberwithin{equation}{section}

\title{\vskip-0.4in Pointwise Bounds and Blow-up for Nonlinear 
Polyharmonic Inequalities}

\author{Steven D.~Taliaferro
\footnote{Mathematics Department,
Texas A\&M University,
College Station, TX 77843-3368,
{\tt stalia@math.tamu.edu},
Phone 979-845-7554,
Fax 979-845-6028}}

\date{} 
\begin{document}
 
\maketitle

\thispagestyle{empty}


\begin{abstract}
  We obtain results for the following question where $m\ge 1$ and
  $n\ge 2$ are integers.  \ms

  \n {\bf Question.} For which continuous functions $f\colon
  [0,\infty)\to [0,\infty)$ does there exist a continuous function
  $\varphi\colon (0,1)\to (0,\infty)$ such that every $C^{2m}$
  nonnegative solution $u(x)$ of
\[
 0 \le -\Delta^m u\le f(u)\quad \text{in}\quad B_2(0)\backslash\{0\}\subset {\bb R}^n
\]
satisfies
\[
 u(x) = O(\varphi(|x|))\quad \text{as}\quad x\to 0
\]
and what is the optimal such $\varphi$ when one exists?
\medskip

\[ \text{\bf R\'esum\'e} \] 

Nous obtenons des r\'esultats pour la question
suivante, avec $m\ge1$ et $n\ge2$ entiers.  \ms
  
\n {\bf Question.} Pour quelles fonctions continues $f\colon
  [0,\infty)\to [0,\infty)$ existe-t-il une fonction continue
  $\varphi\colon (0,1)\to (0,\infty)$ telle que chaque solution $C^{2m}$
  non-negative $u(x)$ de
\[
 0 \le -\Delta^m u\le f(u)\quad \text{dans}\quad B_2(0)\backslash\{0\}\subset {\mathbb R}^n
\]
satisfasse \`a
\[
 u(x) = O(\varphi(|x|))\quad \text{lorsque}\quad x\to 0,
\]
et quelle est la meilleure de ces fonctions $\varphi$ quand elle existe?
\medskip

\n {\it Keywords}: Isolated singularity; Polyharmonic; Blow-up;
Pointwise bound.
\smallskip

\n {\it 2010 Mathematics Subject Classification Codes}: 35B09, 35B33,
35B40, 35B44, 35B45, 35R45, 35J30, 35J91

\end{abstract}
\section{Introduction}\label{sec1.1}

In this paper we consider the following question where $m\ge 1$ and
$n\ge 2$ are integers.  \ms

\n {\bf Question 1.} For which continuous functions $f\colon
[0,\infty)\to [0,\infty)$ does there exist a continuous function
$\varphi\colon (0,1)\to (0,\infty)$ such that every $C^{2m}$
nonnegative solution $u(x)$ of
\begin{equation}\label{eq1.1}
 0 \le -\Delta^m u\le f(u)\quad \text{in}\quad 
  B_2(0)\backslash\{0\}\subset {\bb R}^n
\end{equation}
satisfies
\begin{equation}\label{eq1.2}
 u(x) = O(\varphi(|x|))\quad \text{as}\quad x\to 0
\end{equation}
and what is the optimal such $\varphi$ when one exists?
\ms

\n We call a function $\varphi$ with the above properties a pointwise a priori bound (as $x\to 0$) for $C^{2m}$ nonnegative solutions $u(x)$ of \eqref{eq1.1}.

As we shall see, when $\varphi$ in Question~1 is optimal, the estimate
\eqref{eq1.2} can sometimes be sharpened to
\[
 u(x) = o(\varphi(|x|))\quad \text{as}\quad x\to 0.
\]

\begin{rem}\label{rem1.1}
Let
\begin{equation}\label{eq1.3}
 \Gamma(r) = \begin{cases}
              r^{-(n-2)},&\text{if $n\ge 3$;}\\
\log \frac5r,&\text{if $n=2$.}
             \end{cases}
\end{equation}
Since $u(x)=\Gamma(|x|)$ is a positive solution of $-\Delta^mu = 0$ in
$B_2(0) \backslash\{0\}$, and hence a positive solution of
\eqref{eq1.1}, any pointwise a priori bound $\varphi$ for $C^{2m}$
nonnegative solutions $u(x)$ of \eqref{eq1.1} must be at least as
large as $\Gamma$, and whenever $\varphi=\Gamma$ is such a bound it is
necessarily an optimal bound.
\end{rem}

Some of our results for Question~1 can be generalized to allow the
function $f$ in \eqref{eq1.1} to depend nontrivially on $x$ and the
partial derivatives of $u$ up to order $2m-1$.
(See the second paragraph after Proposition \ref{pro1.1}.)

We also consider the following analog of Question~1 when the
singularity is at $\infty$ instead of at the origin.\medskip

\n {\bf Question 2.} For which continuous functions $f\colon [0,\infty)\to [0,\infty)$ does there exist a continuous function $\varphi\colon  (1,\infty)\to (0,\infty)$ such that every $C^{2m}$ nonnegative solution $v(y)$ of
\begin{equation}\label{eq1.25}
0 \le -\Delta^m v\le f(v)\quad \text{in}\quad {\bb R}^n\backslash B_{1/2}(0)
\end{equation}
satisfies
\[
v(y) = O(\varphi(|y|))\quad \text{as}\quad |y|\to \infty
\]
and what is the optimal such $\varphi$ when one exists?
\medskip

The $m$-Kelvin transform of a function $u(x)$, 
$x\in\Omega\subset {\bb R}^n\backslash\{0\}$, is defined by
\begin{equation}\label{eq1.26}
v(y) = |x|^{n-2m} u(x)\quad \text{where}\quad x=y/|y|^2.
\end{equation}
By direct computation, $v(y)$ satisfies
\begin{equation}\label{eq1.27}
\Delta^mv(y) = |x|^{n+2m} \Delta^mu(x).
\end{equation}
See \cite[p.~221]{WX}  or \cite[p.~660]{X}.
Using this fact and some of our results for Question~1, we will obtain
results for Question~2.

Nonnegative solutions in a punctured neighborhood of the origin in
${\bb R}^n$---or near $x=\infty$ via the $m$-Kelvin transform---of
problems of the form
\begin{equation}\label{eq1.33}
 -\Delta^m u = f(x,u)\quad\text{or}\quad 0\le - \Delta^mu \le f(x,u)
\end{equation}
when $f$ is a nonnegative function have been studied in
\cite{CMS,CX,GL,H,MR,WX,X} and elsewhere. These problems arise
naturally in conformal geometry and in the study of the Sobolev
embedding of $H^{2m}$ into $L^{\frac{2n}{n-2m}}$.

Pointwise estimates at $x=\infty$ of solutions $u$ of problems
\eqref{eq1.33} can be crucial for proving existence results for entire
solutions of \eqref{eq1.33} which in turn can be used to obtain, via
scaling methods, existence and estimates of solutions of boundary
value problems associated with \eqref{eq1.33},
see e.g. \cite{RW1,RW2}. 
An excellent reference for polyharmonic boundary
value problems is \cite{GGS}.

Also, weak solutions of $\Delta^mu = \mu$, where
$\mu$ is a measure on a subset of ${\bb R}^n$, have been studied in
\cite{CDM,FKM,FM}, and removable isolated singularities of 
$\Delta^mu=0$ have been studied in \cite{H}.

Our proofs require Riesz potential estimates as stated, for example,
in \cite[Lemma 7.12]{GT} and a representation formula for $C^{2m}$
nonnegative solutions of 
\begin{equation}\label{eq2.5}
 -\Delta^m u\ge 0\quad \text{in}\quad B_2(0)\backslash\{0\} \subset {\bb R}^n,
\end{equation}
which we state in Lemma \ref{lem2.1}. 

\section{Results for Question 1}\label{sec1.2}

In this section we state and discuss our results for Question 1.
If $m\ge 1$ and $n\ge 2$ are integers then $m$ and $n$ satisfy one of
the following five conditions.
\begin{itemize}
 \item[(i)] either $m$ is even or $2m>n$;
\item[(ii)] $m=1$ and $n\ge 3$;
\item[(iii)] $m=1$ and $n=2$;
\item[(iv)] $m\ge 3$ is odd and $2m<n$;
\item[(v)] $m\ge 3$ is odd and $2m=n$.
\end{itemize}

The following three theorems, which we proved in \cite{GMT},
\cite{T2}, and \cite{T1}, completely answer Question~1 when $m$ and
$n$ satisfy either (i), (ii), or (iii). Consequently, in this paper,
we will only prove results dealing with the
case that $m$ and $n$ satisfy either (iv) or (v).

\begin{thm}\label{thm1.1}
  Suppose $m\ge 1$ and $n\ge 2$ are integers satisfying (i) and
  $f\colon [0,\infty)\to [0,\infty)$ is a continuous function. Let
  $u(x)$ be a $C^{2m}$ nonnegative solution of \eqref{eq1.1} or, more
  generally, of \eqref{eq2.5}.
Then
\begin{equation}\label{eq1.4}
 u(x) = O(\Gamma(|x|))\quad \text{as}\quad x\to 0,
\end{equation}
where $\Gamma$ is given by \eqref{eq1.3}.
\end{thm}

\begin{thm}\label{thm1.2}
Let $u(x)$ be a $C^2$ nonnegative solution of \eqref{eq1.1} where
the integers $m$ and $n$ satisfy (ii), (resp. (iii)), 
and $f\colon [0,\infty)\to
[0,\infty)$ is a continuous function satisfying
\begin{equation}\label{eq1.5}
 f(t) = O(t^{n/(n-2)}), \quad (\text{resp. } \log(1+f(t)) = O(t)) \quad \text{as}\quad t\to \infty.
\end{equation}
Then $u$ satisfies \eqref{eq1.4}.
\end{thm}

By Remark~\ref{rem1.1} the bound \eqref{eq1.4} for $u$ in Theorems~\ref{thm1.1} and \ref{thm1.2} is optimal.

By the following theorem, the condition \eqref{eq1.5} on $f$ in Theorem~\ref{thm1.2} for the existence of a pointwise bound for $u$ is essentially optimal.

\begin{thm}\label{thm1.3}
Suppose $m$ and $n$ are integers satisfying (ii), (resp. (iii)),
and $f\colon [0,\infty)\to [0,\infty)$ is a continuous function satisfying
\begin{equation}\label{eq1.6}
 \lim_{t\to\infty} \frac{f(t)}{t^{n/(n-2)}} = \infty,\quad \left(\text{resp. } \lim_{t\to\infty} \frac{\log(1+f(t))}t = \infty\right).
\end{equation}
Then for each continuous function $\varphi\colon (0,1)\to (0,\infty)$
there exists a $C^2$ positive solution $u(x)$ of \eqref{eq1.1}
such that
\[
 u(x) \ne O(\varphi(|x|))\quad \text{as}\quad x\to 0.
\]
\end{thm}

If $m$ and $n$ satisfy (i), (ii), or (iii), then according to
Theorems~\ref{thm1.1}, \ref{thm1.2}, and \ref{thm1.3}, either the
optimal pointwise bound for $u$ is given by \eqref{eq1.4} or there
does not exists a pointwise bound for $u$, (provided we don't allow
the rather uninteresting and pathological possibility when 
$m$ and $n$ satisfy (ii), (resp. (iii)),
that $f$ satisfies neither \eqref{eq1.5} nor \eqref{eq1.6}).

The situation is very different and more interesting when $m$ and $n$
satisfy (iv) or (v). In this case, according to the following results,
there are an infinite number of different optimal pointwise bounds for
$u$ depending on $f$.

The following three theorems deal with Question~1 when $m$ and $n$ satisfy (iv).

\begin{thm}\label{thm1.4}
Let $u(x)$ be a $C^{2m}$ nonnegative solution of \eqref{eq1.1} where the integers $m$ and $n$ satisfy (iv) and $f\colon [0,\infty)\to [0,\infty)$ is a continuous function satisfying
\[
 f(t) = O(t^\lambda)\quad \text{as}\quad t\to\infty
\]
where
\[
 0 \le \lambda \le \frac{2m+n-2}{n-2},\quad \left(\text{resp. } \frac{2m+n-2}{n-2} < \lambda < \frac{n}{n-2m}\right).
\]
Then as $x\to 0$,
\begin{align}\label{eq1.7}
 u(x) &= O(|x|^{-(n-2)}),\\
\label{eq1.8}
\Big(\text{resp. } u(x) &= o(|x|^{-a})\quad \text{where } 
a=\frac{4m(m-1)}{n-\lambda(n-2m)}\Big).
\end{align}
\end{thm}

Since $a$ in \eqref{eq1.8} is also given by 
\begin{equation}\label{eq1.8-2}
a=n-2+\frac{\lambda(n-2)-(2m+n-2)}{n-\lambda(n-2m)}(n-2m)
\end{equation}
we see that $a$ increases from $n-2$ to infinity as $\lambda$ increases
from $\frac{2m+n-2}{n-2}$ to $\frac{n}{n-2m}$.

By Remark~\ref{rem1.1}, the bound \eqref{eq1.7} is optimal and by the following theorem so is the bound \eqref{eq1.8}.

\begin{thm}\label{thm1.5}Suppose $m$ and $n$ are integers satisfying (iv) and $\lambda$ and $a$ are constants satisfying
\begin{equation}\label{eq1.9}
 \frac{2m+n-2}{n-2} < \lambda < \frac{n}{n-2m}\quad \text{and}\quad a = \frac{4m(m-1)}{n-\lambda(n-2m)}.
\end{equation}
Let $\varphi\colon (0,1)\to (0,1)$ be a continuous function satisfying $\lim\limits_{r\to 0^+} \varphi(r) = 0$. Then there exists a $C^\infty$ positive solution $u(x)$ of
\begin{equation}\label{eq1.10}
 0 \le -\Delta^m u \le u^\lambda\quad \text{in}\quad {\bb R}^n\backslash\{0\}
\end{equation}
such that
\begin{equation}\label{eq1.11}
 u(x)\ne O\left(\varphi(|x|)|x|^{-a}\right)\quad \text{as}\quad x\to 0.
\end{equation}
\end{thm}

With regard to Theorem~\ref{thm1.4}, it is natural to ask what happens when $\lambda\ge \frac{n}{n-2m}$. The answer, given by the following theorem, is that the solutions $u$ can be arbitrarily large as $x\to 0$.

\begin{thm}\label{thm1.6}
  Suppose $m$ and $n$ are integers satisfying (iv) and $\lambda\ge
  \frac{n}{n-2m}$ is a constant. Let $\varphi\colon (0,1)\to
  (0,\infty)$ be a continuous function satisfying $\lim\limits_{r\to
    0^+} \varphi(r)=\infty$. Then there exists a $C^\infty$ positive
  solution $u(x)$ of \eqref{eq1.10} such that
\[
 u(x)\ne O(\varphi(|x|)) \quad \text{as}\quad x\to 0.
\]
\end{thm}

The following five theorems deal with Question~1 when $m$ and $n$ satisfy (v). This is the most interesting case.

\begin{thm}\label{thm1.7}
Let $u(x)$ be a $C^{2m}$ nonnegative solution of \eqref{eq1.1} where the integers $m$ and $n$ satisfy (v) and $f\colon [0,\infty) \to [0,\infty)$ is a continuous function satisfying
\[
 f(t) = O(t^\lambda)\quad \text{as}\quad t\to\infty
\]
where
\[
 0 \le \lambda\le \frac{2n-2}{n-2},\quad \left(\text{resp. } \lambda > \frac{2n-2}{n-2}\right).
\]
Then as $x\to 0$,
\begin{align}\label{eq1.13}
 u(x) &= O(|x|^{-(n-2)}),\\
\label{eq1.14}
\Big(\text{resp. } u(x) &= o\Big(|x|^{-(n-2)} \log \frac5{|x|}\Big)\Big).
\end{align}
\end{thm}

By Remark~\ref{rem1.1}, the bound \eqref{eq1.13} is optimal and by the following theorem so is the bound \eqref{eq1.14}.

\begin{thm}\label{thm1.8}
Suppose $m$ and $n$ are integers satisfying (v) and $\lambda$ is a constant satisfying
\begin{equation}\label{eq1.15}
 \lambda > \frac{2n-2}{n-2}.
\end{equation}
Let $\varphi\colon (0,1)\to (0,1)$ be a continuous function satisfying
$\lim\limits_{r\to0^+} \varphi(r) = 0$. Then there exists a $C^\infty$
positive solution $u(x)$ of \eqref{eq1.10} such that
\begin{equation}\label{eq1.17}
 u(x) \ne O\left(\varphi(|x|)|x|^{-(n-2)} \log \frac5{|x|}\right)\quad \text{as}\quad x\to 0.
\end{equation}
\end{thm}

By the following theorem $u(x)$ may satisfy a pointwise a priori bound even when $f(t)$ grows, as $t\to\infty$, faster than any power of $t$.

\begin{thm}\label{thm1.9}
Let $u(x)$ be a $C^{2m}$ nonnegative solution of \eqref{eq1.1} where the integers $m$ and $n$ satisfy (v) and $f\colon [0,\infty)\to [0,\infty)$ is a continuous function satisfying
\[
 \log(1+f(t)) = O(t^\lambda)\quad \text{as}\quad t\to\infty
\]
where
\begin{equation}\label{eq1.18}
 0 < \lambda < 1.
\end{equation}
Then
\begin{equation}\label{eq1.19}
 u(x) = o\left(|x|^{\frac{-(n-2)}{1-\lambda}}\right)\quad \text{as}\quad x\to 0.
\end{equation}
\end{thm}

By the following theorem, the estimate \eqref{eq1.19} in Theorem~\ref{thm1.9} is optimal.

\begin{thm}\label{thm1.10}
Suppose $m$ and $n$ are integers satisfying (v) and $\lambda$ is a constant satisfying \eqref{eq1.18}. Let $\varphi\colon (0,1)\to (0,1)$ be a continuous function satisfying $\lim\limits_{r\to 0^+} \varphi(r) = 0$. Then there exists a $C^\infty$ positive solution $u(x)$ of
\begin{equation}\label{eq1.20}
 0 \le -\Delta^m u \le e^{u^\lambda}\quad \text{in}\quad {\bb R}^n\backslash\{0\}
\end{equation}
such that
\begin{equation}\label{eq1.21}
 u(x)\ne O\left(\varphi(|x|) |x|^{\frac{-(n-2)}{1-\lambda}}\right)\quad \text{as}\quad x\to 0.
\end{equation}
\end{thm}

With regard to Theorem~\ref{thm1.9}, it is natural to ask what happens
when $\lambda\ge 1$. The answer, given by the following theorem, is
that the solutions $u$ can be arbitrarily large as $x\to 0$.

\begin{thm}\label{thm1.11}
  Suppose $m$ and $n$ are integers satisfying (v) and $\lambda\ge 1$
  is a constant. Let $\varphi\colon (0,1)\to (0,\infty)$ be a
  continuous function satisfying $\lim\limits_{r\to 0^+} \varphi(r) =
  \infty$. Then there exists a $C^\infty$ positive solution of
  \eqref{eq1.20}
such that
\begin{equation}\label{eq1.23}
 u(x) \ne O(\varphi(|x|))\quad \text{as}\quad x\to 0.
\end{equation}
\end{thm}

Theorems~\ref{thm1.3}--\ref{thm1.11} are ``nonradial''. By this we
mean that if one requires the solutions $u(x)$ in Question~1 to be
radial then, according to the following proposition, the complete
answer to Question~1 is very different.

\begin{pro}\label{pro1.1}
  Suppose $m\ge 1$ and $n\ge 2$ are integers and $f\colon [0,\infty)
  \to [0,\infty)$ is a continuous function. Let $u(x)$ be a $C^{2m}$
  nonnegative $\underline{\text{radial}}$ solution of \eqref{eq1.1}
  or, more generally, of \eqref{eq2.5}.
 Then $u$ satisfies \eqref{eq1.4}.
\end{pro}

By Remark~\ref{rem1.1}, the bound \eqref{eq1.4} for $u$ in
Proposition~\ref{pro1.1} is optimal.

Theorems \ref{thm1.4} and \ref{thm1.7} are special cases of much more
general results, in which, instead of obtaining pointwise upper bounds
(when they exist) for $u$ where $u$ is a nonnegative solution of
\[
0 \le -\Delta^m u\le (u+1)^\lambda\quad \text{in}\quad B_2(0)\backslash\{0\},
\]
we obtain pointwise upper bounds (when they exist) for $|D^i u|$,
$i=0,1,2,\ldots, 2m-1$, where $u$ is a solution of
\[
0\le -\Delta^mu \le \sum^{2m-1}_{k=0} 
|x|^{-\alpha_k}(|D^ku|+g_k(x))^{\lambda_k} 
\quad \text{in}\quad B_2(0)\backslash \{0\}\subset {\bb R}^n
\]
such that
\[ 
|x|^{n-2}u(x) \quad \text{is bounded below in $B_2(0) 
\backslash \{0\}$},
\]
where the functions $g_k(x)$ tend to
infinity as $x\to 0$.  See Theorems~\ref{thm3.1} and \ref{thm3.2} in
Section~\ref{sec3} for the precise statements of these more general
results.

Estimates for some derivatives of nonnegative solutions of
\eqref{eq1.1} when $m$ and $n$ satisfy (i) were obtained in
\cite{GMT}.

If $m\ge 1$ and $n\ge 2$ are integers satisfying (i) then, according
to Theorem \ref{thm1.1}, $u$ satisfies a pointwise upper bound as
$x\to 0$ without imposing an upper bound $f(u)$ on $-\Delta^mu$. On
the other hand, if $m$ and $n$ do not satisfy (i) then according to
Theorems \ref{thm1.2}--\ref{thm1.11}, $u$ satisfies a pointwise upper
bound as $x\to 0$ if and only if an appropriate upper bound $f(u)$ is
placed on $-\Delta^m u$. This is due to the following two reasons.
\begin{enumerate}
\item According to formulas \eqref{eq2.1}--\eqref{eq2.3} for the
  fundamental solution $\Phi$ of $\Delta^m$ in ${\bb R}^n$, $\Phi$ is
  bounded below in $B_1(0)\backslash\{0\}$ if and only if $m$ and $n$
  satisfy (i).
\item There is a term in a decomposed version of the representation
  formula \eqref{eq2.6} for nonnegative solutions $u$ of \eqref{eq2.5}
  which is bounded above when $\Phi$ is bounded below. However, when
  $\Phi$ is not bounded below, one needs an upper bound on $-\Delta^m
  u$ to estimate this term. The crux of many of the proofs consists of
  obtaining this estimate.
\end{enumerate}

The term referred to in 2 can be thought of as the convolution
\begin{equation}\label{eqnew1.3}
\intl_{|y|<1} \Phi(x-y)\Delta^m u(y)\,dy.
\end{equation}
However it may happen when $m\ge 2$ that $-\Delta^m u\not\in
L^1(B_1(0))$, in which case this convolution is not finite for every
$x\in {\bb R}^n$. This difficulty is overcome in Lemma \ref{lem2.1} by
replacing $\Phi(x-y)$ in \eqref{eqnew1.3} with the difference of
$\Phi(x-y)$ and a partial sum of the Taylor series of $\Phi$ at $x$.

\section{Results for Question 2}\label{sec1.3}

In this section we state our results for Question~2.

As noted in \cite{GMT}, by applying the $m$-Kelvin transform
\eqref{eq1.26} to the function $u$ in Theorem \ref{thm1.1}, we
immediately obtain the following result concerning Question 2 when $m$
and $n$ satisfy condition (i) at the beginning of Section \ref{sec1.2}.

\begin{thm}\label{thm1.13} 
  Suppose $m\ge 1$ and $n\ge 2$ are integers satisfying (i) and
  $f\colon [0,\infty)\to [0,\infty)$ is a continuous function. Let
  $v(y)$ be a $C^{2m}$ nonnegative solution of \eqref{eq1.25} or, more
  generally, of
\[
 -\Delta^m v\ge 0\quad \text{in}\quad {\bb R}^n\backslash B_{1/2}(0).
\]
Then
\begin{equation}\label{eq1.28}
 v(y) = O(\Gamma_\infty(|y|))\quad \text{as}\quad |y|\to \infty,
\end{equation}
where 
\[
 \Gamma_{\infty}(r) = \begin{cases}
              r^{2m-2},&\text{if $n\ge 3$;}\\
r^{2m-2}\log 5r,&\text{if $n=2$.}
             \end{cases}
\]
\end{thm}

The estimate \eqref{eq1.28} is optimal because $\Delta^m\Gamma_\infty(|y|)=0$
in ${\bb R}^n\backslash\{0\}$.

Using the $m$-Kelvin transform and Theorems~\ref{thm3.1},
\ref{thm3.2}, and \ref{thm3.3} in Section~\ref{sec3} we will prove in
Section~\ref{sec4} the following three theorems dealing with
Question~2, the first of which deals with the case that $m$ and $n$
satisfy condition (iv) at the beginning of Section \ref{sec1.2}.

\begin{thm}\label{thm1.14}
  Let $v(y)$ be a $C^{2m}$ nonnegative solution of \eqref{eq1.25}
  where the integers $m$ and $n$ satisfy $($iv$)$ and $f\colon
  [0,\infty)\to [0,\infty)$ is a continuous function satisfying
\[
f(t) = O(t^\lambda)\quad \text{as}\quad t\to\infty
\]
where 
\[
0 < \lambda < \frac{n}{n-2m}.
\]
Then
\begin{equation}\label{eq1.29}
v(y) = o(|y|^a) \quad \text{as}\quad |y|\to\infty
\end{equation}
where
\[
a=\frac{2m(n-2)}{n-\lambda(n-2m)}
 = 2m-2+ \frac{2(1+\lambda(m-1))}{n-\lambda(n-2m)}(n-2m).
\]
\end{thm}

The next two theorems deal with Question~2 when $m$ and $n$ satisfy 
condition (v) at the beginning of Section \ref{sec1.2}.

\begin{thm}\label{thm1.15}
  Let $v(y)$ be a $C^{2m}$ nonnegative solution of \eqref{eq1.25}
  where the integers $m$ and $n$ satisfy $($v$)$ and $f\colon
  [0,\infty)\to [0,\infty)$ is a continuous function satisfying
\[
f(t) = O(t^\lambda)\quad \text{as}\quad t\to\infty
\]
where $\lambda>0$. Then
\begin{equation}\label{new}
v(y) = o(|y|^{n-2} \log 5|y|) \quad \text{as}\quad |y|\to\infty.
\end{equation}
\end{thm}

\begin{thm}\label{thm1.16}
  Let $v(y)$ be a $C^{2m}$ nonnegative solution of \eqref{eq1.25}
  where the integers $m$ on $n$ satisfy $($v$)$ and $f\colon
  [0,\infty)\to[0,\infty)$ is a continuous function satisfying
\[
\log(1+f(t)) = O(t^\lambda)\quad \text{as}\quad t\to\infty
\]
where $0<\lambda<1$. Then
\[
v(y) = o(|y|^{\frac{n-2}{1-\lambda}}) \quad \text{as}\quad |y|\to\infty.
\]
\end{thm}

Theorems \ref{thm1.14}--\ref{thm1.16} are optimal for Question~2 in
the same way that Theorems~\ref{thm1.4}, \ref{thm1.7}, and
\ref{thm1.9} are optimal for Question~1. For example, according to the
following theorem, the bound \eqref{eq1.29} in Theorem~\ref{thm1.14}
is optimal. We will omit the precise statements and proofs of the other
optimality results for Theorems~\ref{thm1.14}--\ref{thm1.16}.

\begin{thm}\label{thm1.17}
  Suppose $m$ and $n$ are integers satisfying (iv) and $\lambda$ and
  $a$ are constants satisfying
\begin{equation}\label{eq1.30}
0 < \lambda < \frac{n}{n-2m}\quad \text{and}\quad 
a= \frac{2m(n-2)}{n-\lambda(n-2m)}.
\end{equation}
Let $\varphi\colon (1,\infty)\to (0,1)$ be a continuous function
satisfying $\lim\limits_{r\to\infty} \varphi(r) = 0$. Then there
exists a $C^\infty$ positive solution $v(y)$ of
\[
0 \le - \Delta^m v \le v^\lambda\quad \text{in}\quad {\bb R}^n\backslash\{0\}
\]
such that
\[
v(y) \ne O(\varphi(|y|) |y|^a)\quad \text{as}\quad |y|\to\infty.
\]
\end{thm}

See \cite[Corollary 2.5]{T2} for the optimal result concerning Question 2 when
$m$ and $n$ satisfy (iii). We have no results for Question 2 when $m$
and $n$ satisfy (ii), but see \cite{AS} for some related results.

\section{Preliminary results}\label{sec2}

A fundamental solution of $\Delta^m$ in ${\bb R}^n$, where $m\ge 1$
and $n\ge 2$ are integers, is given by
\begin{numcases}{\Phi(x):=A}
(-1)^m |x|^{2m-n}, & if $2 \le 2m < n$; \label{eq2.1}\\
(-1)^{\frac{n-1}2}|x|^{2m-n}, & if $3\le n < 2m$ and  $n$ is odd; \label{eq2.2}\\
(-1)^{\frac{n}2} |x|^{2m-n} \log \frac5{|x|}, & if $2\le n \le 2m$ 
 and $n$ is even; \label{eq2.3}
\end{numcases}
where $A=A(m,n)$ is a \emph{positive} constant whose value may change
from line to line throughout this entire paper.
In the sense of distributions, $\Delta^m\Phi = \delta$, where $\delta$ is the Dirac mass at the origin in ${\bb R}^n$. For $x\ne 0$ and $y\ne x$, let
\begin{equation}\label{eq2.4}
 \Psi(x,y) = \Phi(x-y) - \sum_{|\alpha|\le 2m-3} \frac{(-y)^\alpha}{\alpha!} D^\alpha\Phi(x)
\end{equation}
be the error in approximating $\Phi(x-y)$ with the partial sum of degree $2m-3$ of the Taylor series of $\Phi$ at $x$.

The following lemma, which we proved in \cite{GMT}, gives representation formula \eqref{eq2.6} for nonnegative solutions of inequality \eqref{eq2.5}.
See \cite{FKM,FM} for similar results.

\begin{lem}\label{lem2.1}
Let $u(x)$ be a $C^{2m}$ nonnegative solution of \eqref{eq2.5}
where $m\ge 1$ and $n\ge 2$ are integers. Then $\intl_{|y|<1}
|y|^{2m-2}(-\Delta^m u(y))\,dy<\infty$ and
\begin{equation}\label{eq2.6}
 u = N+h + \sum_{|\alpha|\le 2m-2} a_\alpha D^\alpha\Phi\quad
 \text{in}\quad 
B_1(0)\backslash\{0\}
\end{equation}
where $a_\alpha, |\alpha|\le 2m-2$, are constants, $h\in C^\infty(B_1(0))$ is a solution of
\[
 \Delta^m h = 0\quad \text{in}\quad B_1(0),
\]
and 
\[
 N(x) = \intl_{|y|\le 1} \Psi(x,y) \Delta^mu(y)\,dy \quad \text{for}\quad x\ne 0.
\]
\end{lem}

\begin{lem}\label{lem2.2}
Suppose $f$ is locally bounded, nonnegative, and measurable in $\ovl{B_1(0)}\backslash\{0\} \subset {\bb R}^n$ and
\begin{equation}\label{eq2.8}
 \intl_{|y|<1} |y|^{2m-2} f(y)\,dy < \infty
\end{equation}
where $m\ge 2$ and $n\ge 2$ are integers, $m$ is odd, and $2m\le n$. Let
\begin{equation}\label{eq2.9}
 N(x) = \intl_{|y|<1}  - \Psi(x,y) f(y)\,dy \quad \text{for}\quad x\in {\bb R}^n\backslash\{0\}
\end{equation}
where $\Psi$ is given by \eqref{eq2.4}. Then $N\in C^{2m-1}({\bb R}^n\backslash\{0\})$. Moreover when $|\beta|<2m$ and either $2m=n$ and $|\beta|\ne 0$ or $2m<n$ we have
\begin{equation}\label{eq2.10}
 (D^\beta N)(x) = \intl_{\underset{\sst |y|<1}{|y-x|<|x|/2}} - (D^\beta\Phi)(x-y) f(y)\,dy + O(|x|^{2-n-|\beta|}) \quad \text{for}\quad x\ne 0
\end{equation}
and when $2m=n$ we have
\begin{equation}\label{eq2.11}
 N(x) = A \intl_{\underset{\sst |y|<1}{|y-x|<|x|/2}} \left(\log \frac{|x|}{|x-y|}\right) f(y)\,dy + O(|x|^{2-n}) \quad \text{for}\quad x\ne 0.
\end{equation}
\end{lem}

\begin{proof}
 Differentiating \eqref{eq2.4} with respect to $x$ we get
\[
 D^\beta_x\Psi(x,y) = (D^\beta\Phi)(x-y) - \sum_{|\alpha|\le 2m-3} \frac{(-y)^\alpha}{\alpha!} (D^{\alpha+\beta} \Phi)(x) \quad \text{for}\quad x\ne 0 \quad \text{and}\quad y\ne x
\]
and so by Taylor's theorem applied to $D^\beta\Phi$ we have
\begin{equation}\label{eq2.12}
 |D^\beta_x\Psi(x,y)| \le C|y|^{2m-2} |x|^{2-n-|\beta|}\quad \text{for}\quad |y|<\frac{|x|}2
\end{equation}
where in this proof $C=C(m,n,\beta)$ is a positive constant whose value may change from line to line.

Let $\vp\in(0,1)$ be fixed. Then $N=N_1+N_2$ in ${\bb R}^n\backslash\{0\}$ where
\[
 N_1(x) = \intl_{|y|<\vp} - \Psi(x,y) f(y)\,dy\quad \text{and}\quad N_2(x) = \intl_{\vp<|y|<1} - \Psi(x,y) f(y)\,dy.
\]
It follows from \eqref{eq2.8} and \eqref{eq2.12} that $N_1\in C^\infty({\bb R}^n\backslash \ovl{B_{2\vp}(0)})$ and
\[
 (D^\beta N_1)(x)  = \intl_{|y|<\vp} - D^\beta\Psi(x,y) f(y)\,dy\quad \text{for}\quad |x|>2\vp.
\]
Also, by the boundedness of $f$ in $B_1(0)\backslash B_\vp(0)$, $N_2\in C^{2m-1}({\bb R}^n\backslash \ovl{B_{2\vp}(0)})$ and for $|\beta|<2m$ we have
\[
 (D^\beta N_2)(x) = \intl_{\vp<|y|<1} - D^\beta\Psi(x,y) f(y)\,dy \quad \text{for}\quad |x|>2\vp.
\]
Thus since $\vp\in (0,1)$ was arbitrary, we have $N\in C^{2m-1}({\bb R}^n\backslash \{0\})$ and for $|\beta|<2m$ we have
\begin{equation}\label{eq2.13}
 (D^\beta N)(x) = \intl_{|y|<1} - D^\beta_x \Psi(x,y) f(y)\,dy\quad \text{for}\quad x\ne 0.
\end{equation}

\n {\bf Case 1.} Suppose $|\beta|<2m$ and either $2m=n$ and $|\beta|\ne 0$ or $2m<n$. Then for $0<|x|/2 < |y|$ we have
\[
 \left|\sum_{|\alpha|\le 2m-3} \frac{(-y)^\alpha}{\alpha!} D^{\alpha+\beta} \Phi(x)\right| \le C \sum_{|\alpha| \le 2m-3} |y|^{|\alpha|} |x|^{2m-n-|\alpha|-|\beta|} \le C|y|^{2m-2} |x|^{2-n-|\beta|}
\]
and for $0<|x|/2<|y|$ and $|y-x| > |x|/2$ we have
\[
 |(D^\beta\Phi)(x-y)| \le C|x-y|^{2m-n-|\beta|} \le C|x|^{2m-n-|\beta|} \le C|y|^{2m-2} |x|^{2-n-|\beta|}.
\]
Thus \eqref{eq2.8}, \eqref{eq2.12} and \eqref{eq2.13} imply \eqref{eq2.10}.
\medskip 

\n {\bf Case 2.} Suppose $2m=n$. Then for $0<|x|/2<|y|$ we have 
\[
 \left|\sum_{1\le |\alpha|\le 2m-3} \frac{(-y)^\alpha}{\alpha!} D^\alpha\Phi(x)\right| \le C \sum_{1\le |\alpha|\le 2m-3} |y|^{|\alpha|} |x|^{2m-n-|\alpha|} \le C |y|^{2m-2} |x|^{2-n}
\]
and if $0 < |x|/2<|y|$ and $|y-x|>|x|/2$ then using the fact that $|\log z|\le \log 4z$ for $z\ge 1/2$ we have
\begin{align*}
 |-\Phi(x-y)+\Phi(x)| &= A\left|\log \frac{|x-y|}{|x|}\right|\le A \log 4 \frac{|x-y|}{|x|}\\
&\le A \frac{|y|^{n-2}}{|x|^{n-2}} \left(\frac{|x|}{|y|}\right)^{n-2} \log 4\left(1 + \frac{|y|}{|x|}\right)\\
&\le A \frac{|y|^{n-2}}{|x|^{n-2}} \max_{r\ge 1/2} r^{2-n} \log 4(1+r).
\end{align*}
Thus \eqref{eq2.11} follows from \eqref{eq2.8}, \eqref{eq2.9}, and \eqref{eq2.12}.
\end{proof}

\begin{lem}\label{lem2.3}
Suppose $u(x)$ is a $C^{2m}$ nonnegative solution of \eqref{eq2.5},
where $m\ge 2$ and $n\ge 2$ are integers, $m$ is odd, and $2m\le
n$. Let $\{x_j\}^\infty_{j=1} \subset {\bb R}^n$ and
$\{r_j\}^\infty_{j=1} \subset {\bb R}$ be sequences such that
\begin{equation}\label{eq2.14}
 0 < 4|x_{j+1}| \le |x_j| \le 1/2 \quad \text{and}\quad 0 < r_j \le |x_j|/4.
\end{equation}
Define $f_j\colon B_2(0)\to [0,\infty)$ by
\begin{equation}\label{eq2.15}
 f_j(\eta) = |x_j|^{2m-2} r^n_jf(y)\quad \text{where}\quad y=x_j+r_j\eta
\quad\text{and} \quad f=-\Delta^mu. 
\end{equation}
Then
\begin{equation}\label{eq2.16}
 \intl_{|\eta|<2} f_j(\eta)\,d\eta \to 0\quad \text{as}\quad j\to \infty
\end{equation}
and when $|\beta|<2m$ and either $2m=n$ and $|\beta|\ne 0$ or
$2m<n$ we have for $|\xi|<1$ that
\begin{equation}\label{eq2.17}
 \left(\frac{r_j}{|x_j|}\right)^{n-2m+|\beta|} |x_j|^{n-2+|\beta|}
|(D^\beta u)(x_j+r_j\xi)| 
\le C\left(\frac{r_j}{|x_j|}\right)^{n-2m+|\beta|} 
+ \vp_j + \intl_{|\eta|<2} \frac{Af_j(\eta)\,d\eta}{|\xi-\eta|^{n-2m+|\beta|}}
\end{equation}
and when $2m=n$ we have for $|\xi|<1$ that
\begin{equation}\label{eq2.18}
 \frac{|x_j|^{n-2}}{\log \frac{|x_j|}{r_j}} u(x_j+r_j\xi) \le \frac{C}{\log \frac{|x_j|}{r_j}} +\vp_j + \frac1{\log \frac{|x_j|}{r_j}} \intl_{|\eta|<2} A\left(\log \frac5{|\xi-\eta|}\right) f_j(\eta)\,d\eta
\end{equation}
where in \eqref{eq2.17} and \eqref{eq2.18} the constant $A$ depends
only on $m$ and $n$,
the constant $C$ is independent of $\xi$ and $j$, the constants $\vp_j$ are independent of $\xi$, and $\vp_j\to 0$ as $j\to \infty$.
\end{lem}

\begin{proof}
By Lemma~\ref{lem2.1}, $f$ satisfies \eqref{eq2.8} and
for $|\beta|<2m$ we have
\begin{equation}\label{eq2.19}
(D^\beta u)(x) = (D^\beta N)(x) + O(|x|^{2-n-|\beta|})\quad \text{for}\quad 0<|x|\le 3/4
\end{equation}
where $N$ is given by \eqref{eq2.9}.

If
\[
|y-x| < |x|/2,\quad |y-x_j| > 2r_j,\quad \text{and}\quad |x-x_j|<r_j
\]
then
\[
 |x-y| > r_j\quad \text{and}\quad 2|y|>|x| > |x_j| - r_j > |x_j|/2
\]
and thus when $|\beta|<2m$ and either $2m=n$ and $|\beta|\ne 0$ or $2m<n$ we have
\[
 |(D^\beta \Phi)(x-y)|\le \frac{A}{|x-y|^{n-2m+|\beta|}} \le \frac{A}{r^{n-2m+|\beta|}_j} \le \frac{A|y|^{2m-2}}{r^{n-2m+|\beta|}_j|x_j|^{2m-2}}
\]
and when $2m=n$ we have
\[
 \log \frac{|x|}{|x-y|} \le \log \frac{\frac54 |x_j|}{r_j} \le 2\cdot 4^{n-2} \frac{|y|^{n-2}}{|x_j|^{n-2}} \log \frac{|x_j|}{r_j}.
\]
Thus by \eqref{eq2.8} and Lemma~\ref{lem2.2}, when $|\beta|<2m$ and
either $2m=n$ and $|\beta|\ne 0$ or $2m<n$ we have
\begin{align}
 |(D^\beta N)(x)| &\le \intl_{|y-x_j|<2r_j} \frac{Af(y)\,dy}{|x-y|^{n-2m+|\beta|}} + A\frac{\intl_{|y-x|< |x|/2} |y|^{2m-2} f(y)\,dy}{r^{n-2m+|\beta|}_j |x_j|^{2m-2}} + \frac{C}{|x_j|^{n-2+|\beta|}}\notag\\
\label{eq2.20}
&\le \intl_{|y-x_j| <2r_j} \frac{Af(y)\,dy}{|x-y|^{n-2m+|\beta|}}  + \frac{\vp_j}{r^{n-2m+|\beta|}_j|x_j|^{2m-2}} + \frac{C}{|x_j|^{n-2+|\beta|}} \quad \text{for}\quad |x-x_j|<r_j
\end{align}
and when $2m=n$ we have
\begin{align}
 N(x) &\le A \intl_{|y-x_j|<2r_j} \left(\log \frac{|x|}{|x-y|}\right) f(y)\,dy + 2A4^{n-2} \left(\, \intl_{|y-x|<|x|/2} |y|^{n-2}f(y)\,dy\right) \frac{\log\frac{|x_j|}{r_j}}{|x_j|^{n-2}} + \frac{C}{|x_j|^{n-2}}\notag\\
\label{eq2.21}
&\le A \intl_{|y-x_j|<2r_j} \left(\log \frac{|x|}{|x-y|}\right) f(y)\,dy + \vp_j \frac{\log \frac{|x_j|}{r_j}}{|x_j|^{n-2}} + \frac{C}{|x_j|^{n-2}}\quad \text{for}\quad |x-x_j|<r_j
\end{align}
where in \eqref{eq2.20} and \eqref{eq2.21} the constant $A$ depends
only on $m$ and $n$,
the constant $C$ is independent of $x$ and $j$, the constants $\vp_j$ are independent of $x$, and $\vp_j\to 0$ as $j\to\infty$.

For $|\eta| < 2$ and $y$ given by \eqref{eq2.15} we have $|x_j| < 2|y|$. Thus
\begin{align}
 \intl_{|\eta|<2} f_j(\eta)\,d\eta &= \intl_{|y-x_j|<2r_j} |x_j|^{2m-2} f(y)\,dy\notag\\
\label{eq2.22}
&\le 2^{2m-2} \intl_{|y-x_j|<|x_j|/2} |y|^{2m-2} f(y)\,dy \to 0\quad \text{as}\quad j\to\infty
\end{align}
because $f$ satisfies \eqref{eq2.8}.

If $|\beta|<2m$ and either $2m=n$ and $|\beta|\ne 0$ or $2m<n$ 
then by \eqref{eq2.20} and \eqref{eq2.15} we have for
$|\xi|<1$ that 
\begin{align}
\left(\frac{r_j}{|x_j|}\right)^{n-2m+|\beta|} &|x_j|^{n-2+|\beta|}
|(D^\beta N)(x_j+r_j\xi)|\notag\\ &\le C\left(\frac{r_j}{|x_j|}\right)^{n-2m+|\beta|}
+ \vp_j + r^{n-2m+|\beta|}_j |x_j|^{2m-2} \intl_{|\eta|<2} \frac{Af(y)r^n_j\,d\eta}{r^{n-2m+|\beta|}_j|\xi-\eta|^{n-2m+|\beta|}}\notag\\
\label{eq2.23}
&= C\left(\frac{r_j}{|x_j|}\right)^{n-2m+|\beta|} 
+ \vp_j + \intl_{|\eta|<2} \frac{Af_j(\eta)\,d\eta}{|\xi-\eta|^{n-2m+|\beta|}}.
\end{align}

If $2m=n$ and $|\xi|<1$ then by \eqref{eq2.21}, \eqref{eq2.15}, and \eqref{eq2.22} we have
\begin{align}
 \frac{|x_j|^{n-2}}{\log \frac{|x_j|}{r_j}} N(x_j+r_j\xi) &\le \frac{C}{\log \frac{|x_j|}{r_j}} + \vp_j + \frac{|x_j|^{n-2}}{\log \frac{|x_j|}{r_j}} A \intl_{|\eta|<2} \left(\log \frac{5|x_j|}{r_j|\xi-\eta|}\right) |x_j|^{2-n} f_j(\eta)\,d\eta\notag\\
\label{eq2.24}
&\le \frac{C}{\log \frac{|x_j|}{r_j}} + \vp_j + \frac1{\log \frac{|x_j|}{r_j}} \intl_{|\eta|<2} A \left(\log \frac5{|\xi-\eta|}\right) f_j(\eta)\,d\eta.
\end{align}

Inequalities \eqref{eq2.17} and \eqref{eq2.18} now follow from \eqref{eq2.23}, \eqref{eq2.24}, and \eqref{eq2.19}.
\end{proof}

\begin{lem}\label{lem2.4}
Suppose $m\ge 2$ and $n\ge 2$ are integers, $m$ is odd, and $2m\le n$. Let $\psi\colon (0,1)\to (0,1)$ be a continuous function such that $\lim\limits_{r\to 0^+} \psi(r)=0$. Let $\{x_j\}^\infty_{j=1} \subset {\bb R}^n$ be a sequence such that
\begin{equation}\label{eq2.25}
 0 < 4|x_{j+1}| \le |x_j|\le 1/2
\end{equation}
and
\begin{equation}\label{eq2.26}
 \sum^\infty_{j=1} \vp_j<\infty \quad \text{where}\quad \vp_j = \psi(|x_j|). 
\end{equation}
Let $\{r_j\}^\infty_{j=1}\subset {\bb R}$ be a sequence satisfying
\begin{equation}\label{eq2.27}
 0 < r_j \le |x_j|/5.
\end{equation}
Then there exists a positive function $u\in C^\infty({\bb
  R}^n\backslash\{0\})$ and a positive constant $A=A(m,n)$ such that
\begin{alignat}{2}\label{eq2.28}
0 \le &-\Delta^m u \le \frac{\vp_j}{|x_j|^{2m-2}r^n_j} &\quad &\text{in}\quad B_{r_j}(x_j),\\
\label{eq2.29}
&-\Delta^m u(x) =  0 &\quad &\text{in}\quad {\bb R}^n \Big\backslash\left(\{0\} \cup \bigcup^\infty_{j=1} B_{r_j}(x_j)\right),
\end{alignat}
and
\begin{equation}\label{eq2.30}
 u \ge \begin{cases}
\frac{A\vp_j}{|x_j|^{2m-2}r^{n-2m}_j} &\text{in $B_{r_j}(x_j)$ if $2m<n$}\\
\noalign{\medskip}
\frac{A\vp_j}{|x_j|^{n-2}} \log \frac{|x_j|}{r_j}&\text{in
  $B_{r_j}(x_j)$ if $2m=n$.}
       \end{cases}
\end{equation}
\end{lem}

\begin{proof}
Let $\varphi\colon {\bb R}^n\to [0,1]$ be a $C^\infty$ function whose support is $\ovl{B_1(0)}$. Define $\varphi_j\colon {\bb R}^n\to [0,1]$ by 
\[
 \varphi_j(y) = \varphi(\eta)\quad \text{where}\quad y=x_j+r_j\eta.
\]
Then
\begin{equation}\label{eq2.32}
 \intl_{{\bb R}^n} \varphi_j(y)\,dy = \intl_{{\bb R}^n} \varphi(\eta) r^n_j\,d\eta = r^n_jI
\end{equation}
where $I = \intl_{{\bb R}^n} \varphi(\eta)\,d\eta >0$. Let
\begin{equation}\label{eq2.33}
 f = \sum^\infty_{j=1} M_j\varphi_j \quad \text{where}\quad M_j = \frac{\vp_j}{|x_j|^{2m-2}r^n_j}.
\end{equation}
Since the functions $\varphi_j$ have disjoint supports, $f\in C^\infty({\bb R}^n\backslash\{0\})$ and by \eqref{eq2.27}, \eqref{eq2.32}, \eqref{eq2.33}, and \eqref{eq2.26} we have
\begin{align}
 \intl_{{\bb R}^n} |y|^{2m-2} f(y)\, dy &= \sum^\infty_{j=1} M_j \intl_{|y-x_j|<r_j} |y|^{2m-2} \varphi_j(y)\,dy\notag\\
&\le 2^{2m-2} I \sum^\infty_{j=1} M_j|x_j|^{2m-2} r^n_j\notag\\
\label{eq2.34}
&= 2^{2m-2} I \sum^\infty_{j=1} \vp_j<\infty.
\end{align}

Using the fact that
\begin{equation}\label{eq2.35}
 |x-x_j| < r_j \le |x_j|/5\quad \text{implies}\quad B_{r_j}(x_j) \subset B_{\frac{|x|}2} (x),
\end{equation}
we have for $2m<n$, $x=x_j+r_j\xi$, and $|\xi|<1$ that 
\begin{align*}
 \intl_{|y-x|<|x|/2} \frac1{|x-y|^{n-2m}} f(y)\,dy &\ge \intl_{|y-x_j|<r_j} \frac1{|x-y|^{n-2m}} M_j\varphi_j(y)\,dy\\
&= \intl_{|\eta|<1} \frac1{r^{n-2m}_j} \frac{M_j}{|\xi-\eta|^{n-2m}} \varphi(\eta) r^n_j\, d\eta\\
&= \frac{\vp_j}{|x_j|^{2m-2}r^{n-2m}_j} \intl_{|\eta|<1} \frac{\varphi(\eta)}{|\xi-\eta|^{n-2m}} d\eta\\
&\ge \frac{J\vp_j}{|x_j|^{2m-2}r^{n-2m}_j} \quad \text{where}\quad J = \min_{|\xi|\le 1} \intl_{|\eta|<1} \frac{\varphi(\eta)\,d\eta}{|\xi-\eta|^{n-2m}}.
\end{align*}
Similarly, using \eqref{eq2.35} we have for $2m=n$, $x=x_j+r_j\xi$, and $|\xi|<1$ that
\begin{align*}
 \intl_{|y-x|<|x|/2} \left(\log \frac{|x|}{|x-y|}\right) f(y)\,dy &\ge \intl_{|y-x_j|<r_j} \left(\log \frac{|x|}{|x-y|}\right) M_j \varphi_j(y)\,dy\\
&\ge \intl_{|\eta|<1} \left(\log \frac{\frac45 |x_j|}{r_j|\xi-\eta|}\right) M_j\varphi(\eta) r^n_j\, d\eta\\
&= \frac{\vp_j}{|x_j|^{n-2}} \intl_{|\eta|<1} \left(\log \frac2{|\xi-\eta|} + \log \frac{|x_j|}{r_j} - \log \frac52\right) \varphi(\eta)\,d\eta\\
&\ge \frac{I\vp_j}{|x_j|^{n-2}} \log \frac{|x_j|}{r_j} - \frac{I}{|x_j|^{n-2}} \log \frac52.
\end{align*}
Thus defining $N$ by \eqref{eq2.9}, where $f$ is given by \eqref{eq2.33},
 it follows from \eqref{eq2.34} and Lemma~\ref{lem2.2} that there exists a positive constant $C$ independent of $\xi$ and $j$ such that if we define $u\colon {\bb R}^n\backslash\{0\}\to {\bb R}$ by
\[
 u(x) = N(x) + C|x|^{-(n-2)}
\]
then $u$ is a $C^\infty$ positive solution of
\begin{equation}\label{eq2.36}
 -\Delta^m u = f\quad \text{in}\quad {\bb R}^n\backslash\{0\}
\end{equation}
and for some positive constant $A=A(m,n)$, $u$ satisfies \eqref{eq2.30}.

Also, \eqref{eq2.36} and \eqref{eq2.33} imply that $u$ satisfies \eqref{eq2.28} and \eqref{eq2.29}.
\end{proof}

\begin{rem}\label{rem2.1}
Suppose the hypotheses of Lemma~\ref{lem2.4} hold and $u$ is as in Lemma~\ref{lem2.4}.
\end{rem}

\n {\bf Case 1.} Suppose $2m<n$. Then it follows from \eqref{eq2.28}, \eqref{eq2.29}, and \eqref{eq2.30} that $u$ is a $C^\infty$ positive solution of
\[
 0 \le -\Delta^m u \le |x|^\tau u^\lambda\quad \text{in}\quad {\bb
   R}^n\backslash\{0\},\quad \lambda>0,\quad \tau\in {\bb R},
\]
provided
\[
 \frac{\psi(|x_j|)}{|x_j|^{2m-2}r^n_j} \le 2^{-|\tau|}|x_j|^\tau
\left(\frac{A\psi(|x_j|)}{|x_j|^{2m-2} r^{n-2m}_j}\right)^\lambda
\]
which holds if and only if
\begin{equation}\label{eq2.36-2}
 r_j^{n-\lambda(n-2m)}\ge\frac{2^{|\tau|}}{A^\lambda}
\frac{|x_j|^{(\lambda-1)(2m-2)-\tau}}{\psi(|x_j|)^{\lambda-1}}.
\end{equation}

\n {\bf Case 2.} Suppose $2m=n$. Then it follows from \eqref{eq2.28}, \eqref{eq2.29}, and \eqref{eq2.30} that $u$ is a $C^\infty$ positive solution of
\[
 0 \le -\Delta^mu \le f(u)\quad \text{in}\quad {\bb R}^n\backslash\{0\},
\]
where $f\colon [0,\infty)\to [0,\infty)$ is a nondecreasing continuous function, provided
\begin{equation}\label{eq2.37}
 \frac{\psi(|x_j|)}{|x_j|^{n-2}r^n_j} \le f\left(\frac{A\psi(|x_j|)}{|x_j|^{n-2}} \log \frac{|x_j|}{r_j}\right).
\end{equation}
If $f(u) = u^\lambda, \lambda>1$, then \eqref{eq2.37} holds if and only if
\[
 \log \frac{|x_j|}{r_j} \ge \left(\frac{|x_j|}{r_j}\right)^{\frac{n}\lambda} \frac{|x_j|^a}{A\psi(|x_j|)^{\frac{\lambda-1}\lambda}} \quad \text{where}\quad a = \frac{(n-2)(\lambda-1)-n}\lambda.
\]
If $f(u) = e^{u^\lambda}$, $\lambda>0$, then \eqref{eq2.37} holds if and only if
\[
 \log \frac{\psi(|x_j|)}{|x_j|^{2n-2}} + n \log \frac{|x_j|}{r_j} \le \left(\frac{A\psi(|x_j|)}{|x_j|^{n-2}} \log \frac{|x_j|}{r_j}\right)^\lambda.
\]

\begin{lem}\label{lem2.5}
  Suppose $p>1$ and $R\in (0,2)$ are constants and $g\colon {\bb R}^n\to {\bb
    R}$ is defined by
\[
 g(\xi) = \intl_{|\eta|<R} \left(\log \frac5{|\xi-\eta|}\right) f(\eta)\,d\eta
\]
where $f\in L^1(B_R(0))$, (resp. $f\in L^p(B_R(0))$) . Then
\[
 \|g\|_{L^p(B_R(0))} \le C\|f\|_{L^1(B_R(0))}, \quad 
(\text{resp. }\|g\|_{L^\infty(B_R(0))} \le C\|f\|_{L^p(B_R(0))} ),
\]
where $C=C(n,p,R)$ is a positive constant.
\end{lem}

\begin{proof}
Define $p'$ by $\frac1p + \frac1{p'} = 1$. Then by H\"older's inequality we have 
\begin{align*}
\intl_{|\xi|<R} |g(\xi)|^p \,d\xi &\le \intl_{|\xi|<R} \left[\,
  \intl_{|\eta|<R} \left(\log \frac5{|\xi-\eta|}\right)
  |f(\eta)|^{1/p} |f(\eta)|^{1/p'} \,d\eta\right]^p d\xi\notag\\
&\le \intl_{|\xi|<R} \left[\left(\, \intl_{|\eta|<R} \left(\log \frac5{|\xi-\eta|}\right)^p |f(\eta)|\,d\eta\right)^{1/p} \left(\, \intl_{|\eta|<R} |f(\eta)|\,d\eta\right)^{1/p'}\right]^p d\xi\notag\\
&= \left(\, \intl_{|\eta|<R} |f(\eta)|\,d\eta\right)^{p/p'} \intl_{|\eta|<R} \left(\, \intl_{|\xi|<R} \left(\log \frac5{|\xi-\eta|} \right)^p d\xi\right) |f(\eta)|\,d\eta\notag\\
&\le C(n,p,R) \left(\, \intl_{|\eta|<R} |f(\eta)|\,d\eta\right)^p.
\end{align*}
The parenthetical part follows from H\"older's inequality.
\end{proof}

\section{Proofs when the singularity is at the origin}\label{sec3}

In this section we prove Theorems~\ref{thm1.4}--\ref{thm1.11} 
and Proposition \ref{pro1.1} which
deal with the case that the singularity is at the origin. 
Theorem \ref{thm1.4} will follow easily from the following more
general result.

\begin{thm}\label{thm3.1}
  Suppose $u(x)$ is a $C^{2m}$ solution of 
\begin{equation}\label{eq3.1} 
0\le -\Delta^mu \le K\sum^{2m-1}_{k=0} 
|x|^{-\alpha_k}(|D^ku|+g_k(x))^{\lambda_k} 
\quad \text{in}\quad B_2(0)\backslash \{0\}\subset {\bb R}^n
\end{equation}
such that
\begin{equation}\label{eqnew1.1}
|x|^{n-2}u(x) \quad \text{is bounded below in $B_2(0) 
\backslash \{0\}$},
\end{equation}
where $K>0$, $\lambda_k$, and $\alpha_k$ are constants,
$m\ge 2$ and $n\ge 2$ are integers, $m$ is odd, $2m<n$, and for 
$k=0,1,\ldots,2m-1$ we have 
\begin{equation}\label{eq3.2}
 0\le \lambda_k < \frac{n}{n-2m+k} 
\end{equation}
and $g_k:B_2(0)\backslash\{0\}\to [1,\infty)$ is a continuous function.
Let
\[
a_k=(n-2+k)+b(n-2m+k)
\]
where
\begin{equation}\label{eq3.3}
  b = \max\left\{0, \max_{0\le k\le 2m-1} 
    \frac{\alpha_k+\lambda_k(n-2+k) - (2m+n-2)}{n-\lambda_k(n-2m+k)}\right\}.
\end{equation}
\begin{itemize}
 \item[\rm (i)] If for $k=0,1,\ldots,2m-1$ we have 
\begin{equation}\label{eq3.4}
 g_k(x) = O(|x|^{-a_k}) \quad \text{as}\quad x\to 0
\end{equation}
then for $i=0,1,\ldots, 2m-1$ we have
\[
 |D^iu(x)| = O(|x|^{-a_i})\quad \text{as}\quad x\to 0.
\]
\item[\rm (ii)] If $b>0$ and for $k=0,1,\ldots,2m-1$ we have
\begin{equation}\label{eq3.5}
 g_k(x) = o(|x|^{-a_k})\quad \text{as}\quad x\to 0
\end{equation}
and
\begin{equation}\label{al}
\lambda_k> 0
\end{equation}
then for $i=0,1,\ldots, 2m-1$ we have
\[
 |D^iu(x)| = o(|x|^{-a_i})\quad \text{as}\quad x\to 0.
\]
\end{itemize}
\end{thm}

\begin{proof}
It suffices to prove Theorem \ref{thm3.1} when $u$ is
nonnegative. To see this choose $M>0$ such that 
\[
v(x):=u(x)+M|x|^{-(n-2)}>0 \quad\text{for }0<|x|<2,
\]
which is possible by \eqref{eqnew1.1}, and then apply Theorem \ref{thm3.1}
to $v$ after noting that $-\Delta^m v=-\Delta^m u$ and
\[
|D^ku(x)-D^kv(x)|= O(|x|^{-(n-2+k)})=\begin{cases}
O(|x|^{-a_k}) &\text{if $b=0$}\\
o(|x|^{-a_k}) &\text{if $b>0$.}
\end{cases}
\]

Suppose for contradiction that part (i) (resp.\ part (ii)) is
false. Then there exist $i\in \{0,1,2,\ldots, 2m-1\}$ and a sequence
$\{x_j\}^\infty_{j=1} \subset {\bb R}^n$ such that
\[
 0 < 4|x_{j+1}| < |x_j| < 1/2,
\]
and
\begin{align}\label{eq3.7}
 &|x_j|^{a_i} |D^iu(x_j)|\to \infty\quad \text{as}\quad j\to\infty,\\
\label{eq3.8}
&(\text{resp. } \liminf_{j\to\infty} |x_j|^{a_i} |D^iu(x_j)|>0).
\end{align}
Let $r_j =  |x_j|^{b+1}/4$.
Then $x_j$ and $r_j$ satisfy \eqref{eq2.14}. Let $f_j$ be as in
Lemma~\ref{lem2.3}. Since
\begin{equation}\label{eq3.9}
 \frac{r_j}{|x_j|} = \frac{|x_j|^b}{4}
\end{equation}
it follows from \eqref{eq2.17} with $|\beta| = i$ and $\xi=0$ that
\[
 \frac{|x_j|^{n-2+i+b(n-2m+i)}}{4^{n-2m+i}} |D^iu(x_j)| \le C|x_j|^{(n-2m+i)b} + \vp_j + \intl_{|\eta|<2} \frac{Af_j(\eta)d\eta}{|\eta|^{n-2m+i}}.
\]
Hence \eqref{eq3.7} (resp.\ \eqref{eq3.8}) implies
\begin{align}\label{eq3.10}
 &\intl_{|\eta|<2} \frac{f_j(\eta)d\eta}{|\eta|^{n-2m+i}} \to \infty\quad \text{as}\quad j\to\infty\\
\label{eq3.11}
&\left(\text{resp. } \liminf_{j\to\infty} \intl_{|\eta|<2} \frac{f_j(\eta)d\eta}{|\eta|^{n-2m+i}}>0\right).
\end{align}

On the other hand, \eqref{eq2.15}, \eqref{eq3.1}, and \eqref{eq2.17} imply for $|\xi|<1$ that
\begin{align}
 f_j(\xi) &\le |x_j|^{2m+n-2} \left(\frac{r_j}{|x_j|}\right)^n
 K\sum^{2m-1}_{k=0} |x_j+r_j\xi|^{-\alpha_k}
(|D^ku(x_j+r_j\xi)|+g_k(x_j+r_j\xi))^{\lambda_k}\notag\\
&\le C\sum^{2m-1}_{k=0} \frac{|x_j|^{2m+n-2}
  \left(\frac{r_j}{|x_j|}\right)^n |x_j|^{-\alpha_k}}
{\left(|x_j|^{n-2+k} \left(\frac{r_j}{|x_j|}\right)^{n-2m+k}\right)^{\lambda_k}} \bigg(\left(\frac{r_j}{|x_j|}\right)^{n-2m+k} + \vp_j + \intl_{|\eta|<2} \frac{f_j(\eta)d\eta}{|\xi-\eta|^{n-2m+k}}\notag\\
\label{eq3.12}
&\quad +|x_j|^{n-2+k} \left(\frac{r_j}{|x_j|}\right)^{n-2m+k} g_k(x_j+r_j\xi)\bigg)^{\lambda_k},
\end{align}
where $C$ is a constant independent of $\xi$ and $j$ whose value may
change from line to line.
But \eqref{eq3.9} and \eqref{eq3.3} imply
\begin{align*}
 \frac{|x_j|^{2m+n-2}
   \left(\frac{r_j}{|x_j|}\right)^n|x_j|^{-\alpha_k}}{\left(|x_j|^{n-2+k}
     \left(\frac{r_j}{|x_j|}\right)^{n-2m+k} \right)^{\lambda_k}} 
&= |x_j|^{(2m+n-2)-\lambda_k(n-2+k)-\alpha_k} 
\left(\frac{r_j}{|x_j|}\right)^{n-\lambda_k(n-2m+k)}\\
&\le |x_j|^{(2m+n-2)-\lambda_k(n-2+k) -\alpha_k + (n-\lambda_k(n-2m+k))b}\\
&\le 1,\\
|x_j|^{n-2+k} \left(\frac{r_j}{|x_j|}\right)^{n-2m+k} 
&\le |x_j|^{n-2+k+b(n-2m+k)}\\
&= |x_j|^{a_k},\\
\intertext{and}
\left(\frac{r_j}{|x_j|}\right)^{n-2m+k} &\le |x_j|^{b(n-2m+k)}.
\end{align*}
Hence by \eqref{eq3.4}, (resp. \eqref{eq3.5}) and \eqref{eq3.12} we have
\begin{align}\label{eq3.13}
 f_j(\xi) &\le C\sum^{2m-1}_{k=0} \bigg(1 + \intl_{|\eta|<2} \frac{f_j(\eta)d\eta}{|\xi-\eta|^{n-2m+k}}\bigg)^{\lambda_k} \quad \text{for}\quad |\xi|<1,\\
\label{eq3.14}
\bigg(\text{resp. } f_j(\xi) &\le C\sum^{2m-1}_{k=0} \bigg(\vp_j +\intl_{|\eta|<2} \frac{f_j(\eta)d\eta}{|\xi-\eta|^{n-2m+k}}\bigg)^{\lambda_k}\quad \text{for}\quad |\xi|<1\bigg).
\end{align}
Since
\[
 \intl_{2R\le |\eta|<2} \frac{f_j(\eta)d\eta}{|\xi-\eta|^{n-2m+k}} \le \frac{1}{R^{n-2m+k}} \intl_{|\eta|<2} f_j(\eta)d\eta \quad \text{for}\quad |\xi|<R<1
\]
we have by \eqref{eq3.13}, (resp.\ \eqref{eq3.14}), and \eqref{eq2.16} that 
\begin{equation}\label{eq3.14-2}
 f_j(\xi) \le C\sum^{2m-1}_{k=0} \left(\frac{1}{R^{n-2m+k}} 
+ \intl_{|\eta|<2R} \frac{f_j(\eta)d\eta}{|\xi-\eta|^{n-2m+k}} \right)^{\lambda_k} \quad \text{for}\quad |\xi|<R\le 1
\end{equation}
where $C$ is independent of $\xi,j$, and $R$, (resp.
\begin{equation}\label{eq3.14-3}
 f_j(\xi) \le C\sum^{2m-1}_{k=0} \left(\frac{\vp_j}{R^{n-2m+k}} 
+ \intl_{|\eta|<2R} \frac{f_j(\eta)d\eta}{|\xi-\eta|^{n-2m+k}}\right)^{\lambda_k}\quad \text{for}\quad |\xi|<R\le 1
\end{equation}
where $\vp_j$ is independent of $\xi$ and $R$ and $\vp_j\to 0$ as $j\to\infty$).

We can assume the $\lambda_k$ in \eqref{eq3.14-2} satisfy, instead of 
\eqref{eq3.2}, the stronger condition 
\begin{equation}\label{eq3.14-4}
0<\lambda_k<\frac{n}{n-2m+k}
\end{equation}
because slightly increasing those $\lambda_k$ in \eqref{eq3.14-2} which
are zero will increase the right side of \eqref{eq3.14-2}. By
\eqref{eq3.2} and \eqref{al} the $\lambda_k$ in \eqref{eq3.14-3}
already satisfy \eqref{eq3.14-4}.

It follows from \eqref{eq3.14-2}, (resp. \eqref{eq3.14-3}), and Riesz
potential estimates (see \cite[Lemma~7.12]{GT}) that if the functions
$f_j$ are bounded (resp.\ tend to zero) in $L^p(B_{2R}(0))$ for some
$p\ge 1$ and $R\in (0,1]$ then the functions $f_j$ are bounded (resp.\
tend to zero) in $L^q(B_R(0))$, $0<q\le\infty$, provided
\[
\frac1p-\frac1{q\lambda_k} < \frac{2m-k}{n} \quad\text{for}\quad
k=0,1,\ldots, 2m-1,
\]
which holds if and only if 
\[
\frac1p-\frac1q < \min\limits_{0\le k\le
  2m-1}\left(\frac{(2m-k)\lambda_k}{n}-\frac{\lambda_k-1}{p}\right).
\]
However, 
\[
\inf_{p\ge 1}\frac{(2m-k)\lambda_k}{n}-\frac{\lambda_k-1}{p}
=\min\left\{\frac{n-\lambda_k(n-2m+k)}{n},\frac{\lambda_k(2m-k)}{n}\right\}
\]
which, by \eqref{eq3.14-4}, is bounded below by some positive constant
independent of $k$. 
So starting with \eqref{eq2.16} and iterating the above $L^p$ to $L^q$
comment a finite number of times, we see that there exists $R_0\in
(0,1)$ such that the functions $f_j$ are bounded (resp.\ tend to zero)
in $L^\infty(B_{R_0}(0))$ which together with \eqref{eq2.16}
contradicts \eqref{eq3.10} (resp.\ \eqref{eq3.11}) and thereby
completes the proof of Theorem~\ref{thm3.1}.
\end{proof}

\begin{proof}[Proof of Theorem \ref{thm1.4}]
For some constant $K>0$, $u$ satisfies
\begin{equation}\label{simple}
0\le \Delta^mu\le K(u^\lambda+1) \quad \text{in}\quad 
B_2(0)\backslash \{0\}\subset {\bb R}^n. 
\end{equation}
Thus $u$ satisfies \eqref{eq3.1} with $\lambda_0=\lambda$,
$\alpha_0=0$, $g_0(x)\equiv 1$, and
$\lambda_k=1$, $\alpha_k=-(n-2+k)$, $g_k(x)\equiv 1$ for
$k=1,2,\ldots,2m-1$. 

Let $b$ and $a_k$ be as in Theorem \ref{thm3.1}. Then
\[
b=0, \quad 
\left(\text{resp. } b=\frac{\lambda(n-2)-(2m+n-2)}{n-\lambda(n-2m)}>0\right)
\]
and so $a_0=n-2$, (resp. $a_0=a$, where $a$ is given by
\eqref{eq1.8-2}). Hence \eqref{eq1.7}, (resp. \eqref{eq1.8}), follows from
part (i), (resp. part (ii)), of Theorem \ref{thm3.1}.
\end{proof}

\begin{proof}[Proof of Theorem \ref{thm1.5}]
Define $\psi\colon (0,1)\to (0,1)$ by
\begin{equation}\label{eq3.15}
\psi(r) = \max\left\{\varphi(r)^p, r^{\frac{n-\lambda(n-2m)}{\lambda-1}\frac{b}2}\right\}
\end{equation}
where
\[
b :=  \frac{\lambda(n-2) -(2m+n-2)}{n-\lambda(n-2m)}\quad \text{and}\quad p := \frac{n-\lambda(n-2m)}{4m}
\]
are positive by \eqref{eq1.9}. Let $\{x_j\}^\infty_{j=1} \subset {\bb
  R}^n$ be a sequence satisfying \eqref{eq2.25} and
\eqref{eq2.26}. Define $r_j>0$ by \eqref{eq2.36-2} with the greater than sign
replaced with an equal sign and with $\tau=0$. Then by \eqref{eq3.15}
\begin{align}\label{eq3.16}
r_j &= A^{\frac{-\lambda}{n-\lambda(n-2m)}} \frac{|x_j|^{1+b}}{\psi(|x_j|)^{\frac{\lambda-1}{n-\lambda(n-2m)}}}\\
&\le A^{\frac{-\lambda}{n-\lambda(n-2m)}} |x_j|^{1+b/2}.\notag
\end{align}
Thus by taking a subsequence of $j$, $r_j$ will satisfy \eqref{eq2.27}.

Let $u$ be as in Lemma~\ref{lem2.4}. Then by Case~1 of
Remark~\ref{rem2.1}, $u$ is a $C^\infty$ positive solution of
\eqref{eq1.10} and by \eqref{eq2.30}, \eqref{eq3.16}, \eqref{eq3.15},
and \eqref{eq1.8-2} we have
\begin{align*}
u(x_j) &\ge \frac{A\psi(|x_j|)}{|x_j|^{2m-2}} \frac{A^{\frac{\lambda(n-2m)}{n-\lambda(n-2m)}} \psi(|x_j|)^{\frac{(\lambda-1)(n-2m)}{n-\lambda(n-2m)}}}{|x_j|^{(n-2m)(1+b)}}\\
&= C(m,n,\lambda) \frac{\psi(|x_j|)^{\frac{2m}{n-\lambda(n-2m)}}}{|x_j|^{n-2+(n-2m)b}}\\
&\ge C(m,n,\lambda) \frac{\varphi(|x_j|)^{1/2}}{|x_j|^a}
\end{align*}
which implies \eqref{eq1.11}.
\end{proof}

\begin{proof}[Proof of Theorem \ref{thm1.6}]
Define $\psi\colon (0,1)\to (0,1)$ by $\psi(r) = r^{m-1}$. Let $\{x_j\}^\infty_{j=1} \subset {\bb R}^n$ be a sequence satisfying \eqref{eq2.25}, \eqref{eq2.26}, and
\begin{equation}\label{eq3.18}
\frac{1}{A^\lambda}\frac{|x_j|^{(\lambda-1)(2m-2)}}{\psi(|x_j|)^{\lambda-1}}
=A^{-\lambda}|x_j|^{(\lambda-1)(m-1)}<1
\end{equation}
where $A=A(m,n)$ is as in Lemma~\ref{lem2.4}. Let $\{r_j\}^\infty_{j=1} \subset {\bb R}$ be a sequence satisfying \eqref{eq2.27} and
\begin{equation}\label{eq3.19}
 \frac{A\psi(|x_j|)}{|x_j|^{2m-2}r^{n-2m}_j} > \varphi(|x_j|)^2.
\end{equation}
Since $r_j<1$ and $n-\lambda(n-2m)\le 0$, we see that the left side of
\eqref{eq2.36-2} is greater than or equal to one. Thus \eqref{eq3.18}
implies \eqref{eq2.36-2} with $\tau=0$.  Let $u$ be as in
Lemma~\ref{lem2.4}. Then by \eqref{eq2.30} and \eqref{eq3.19}
\[
 \frac{u(x_j)}{\varphi(|x_j|)} \ge \varphi(|x_j|) \to \infty\quad \text{as}\quad j\to\infty
\]
and by Case~1 of Remark~\ref{rem2.1}, $u$ is a $C^\infty$ positive
solution of \eqref{eq1.10}.
\end{proof}

Theorem \ref{thm1.7} will follow easily from the following more
general result.

\begin{thm}\label{thm3.2}
Suppose $u$ is a $C^{2m}$  solution of \eqref{eq3.1} satisfying
  \eqref{eqnew1.1}
where $K>0$, $\lambda_k$, and $\alpha_k$ are constants; $m\ge 2$, and
$n\ge 2$ are integers, $m$ is odd, $2m=n$; 
\begin{equation}\label{l0}
\lambda_0\ge 0 \quad \text{and}\quad 0\le \lambda_k<n/k \quad \text{for}\quad 
k=1,2,\ldots ,n-1;
\end{equation}
and $g_k\colon B_2(0)\backslash\{0\} \to [1,\infty)$ is a continuous
function. Let
\[
 b = \max\{0,b_0,b_1,\ldots, b_{n-1}\} \quad \text{where}\quad 
b_k = \frac{\alpha_k+\lambda_k(n-2+k) - (2n-2)}{n-k\lambda_k}.
\]
\begin{itemize}
 \item[\rm (i)] Suppose as $x\to 0$ we have 
\begin{align}\label{g0}
g_0(x)&=\begin{cases}
O(|x|^{-(n-2)}) &\text{if $b=0$}\\
O(|x|^{-(n-2)}\log\frac{5}{|x|}) &\text{if $b>0$}
\end{cases}\\
\intertext{and for $k=1,2,\ldots,n-1$ we have}
\label{gk}
g_k(x)&=\begin{cases}
O(|x|^{-(n-2+k)}) &\text{if $b=0$}\\
O(|x|^{-(n-2+k)}a(x)^{-k}) &\text{if $b>0$.}
\end{cases}\\
\intertext{where} 
\label{a(x)}
a(x) &= \min\left\{\frac{|x|^{b_0}}{\left(\log\frac5{|x|}\right)^{\lambda_0/n}}, |x|^{b_1},\ldots, |x|^{b_{n-1}}\right\}.
\end{align}
Then as $x\to 0$ we have
\begin{equation}\label{u}
u(x)=\begin{cases}
O(|x|^{-(n-2)}) &\text{if $b=0$}\\
O(|x|^{-(n-2)}\log\frac{5}{|x|}) &\text{if $b>0$}
\end{cases}
\end{equation}
and for $i=1,2,\ldots,n-1$ we have
\begin{equation}\label{Diu}
|D^iu(x)|=\begin{cases}
O(|x|^{-(n-2+i)}) &\text{if $b=0$}\\
O(|x|^{-(n-2+i)}a(x)^{-i}) &\text{if $b>0$.}
\end{cases}
\end{equation}
\item[\rm (ii)] Suppose $b>0$,
\begin{equation}\label{lk}
\lambda_k>0\quad\text{for}\quad k=0,1,\ldots,n-1,
\end{equation}
and as $x\to 0$ we have
\begin{align}\label{g02}
g_0(x) &= o\left(|x|^{-(n-2)} \log \frac{5}{|x|}\right)\\
\intertext{and} 
\label{gk2}
g_k(x) &= o\left(|x|^{-(n-2+k)} a(x)^{-k}\right) \quad\text{for}
\quad k=1,2,\ldots,n-1,
\end{align}
where $a(x)$ is defined by \eqref{a(x)}. Then 
then as $x\to 0$ we have
\begin{align}\label{u2}
 u(x) &= o\left(|x|^{-(n-2)} \log \frac5{|x|}\right)\\
\intertext{and}
\label{Di2}
 |D^iu(x)| &= o(|x|^{-(n-2+i)} a(x)^{-i})\quad \text{for}\quad i=1,2,\ldots, n-1.
\end{align}
\end{itemize}
\end{thm}

\begin{proof}
As in the proof of Theorem \ref{thm3.1}, it suffices to prove Theorem
\ref{thm3.2} when $u$ is nonnegative.

For $b\ge 0$ and $0<|x|<2$ we define
\[
a_b(x)=\begin{cases}
\frac{1}{4} &\text{if $b=0$}\\
a(x) &\text{if $b>0$}
\end{cases} 
\]
where $a(x)$ is given by \eqref{a(x)}. Then 
\begin{equation}\label{loga}
\log\frac{1}{a_0(x)}=\log 4  \quad \text{and} \quad 
\lim_{x\to 0} \frac{\log\frac{1}{a_b(x)}}{\log\frac{5}{|x|}} =b
\quad \text{when}\quad b>0.
\end{equation}
Thus \eqref{u}, \eqref{Diu}, \eqref{u2}, and \eqref{Di2} are
equivalent to 
\begin{align}\label{u3}
u(x)&=O\left(|x|^{-(n-2)}\log\frac{1}{a_b(x)}\right)\\ 
\label{Diu3}|D^iu(x)|&=O\left(|x|^{-(n-2+i)}a_b(x)^{-i}\right)\\ 
\label{u4}u(x)&=o\left(|x|^{-(n-2)}\log\frac{1}{a_b(x)}\right)\\ 
\label{Diu4}|D^iu(x)|&=o\left(|x|^{-(n-2+i)}a_b(x)^{-i}\right)
\end{align}
and similarly for  \eqref{g0}, \eqref{gk}, \eqref{g02}, and
\eqref{gk2}.

Suppose for contradiction that part (i), (resp. part (ii)), of Theorem
\ref{thm3.2} is false. Then there exists $i_0\in \{0,1,\ldots,n-1\}$
such that the estimate \eqref{u3}, \eqref{Diu3}, (resp. \eqref{u4},
\eqref{Diu4}), for $D^iu$ does not hold when $i=i_0$. Thus there is a
sequence $\{x_j\}^\infty_{j=1} \subset {\bb R}^n$ satisfying
\begin{equation}\label{xj}
 0 < 4|x_{j+1}| < |x_j| < 1/2\quad \text{and}\quad a_b(x_j)\le \frac14
\end{equation}
such that
\begin{align}\label{limu}
 \liminf_{j\to\infty} \frac{|x_j|^{n-2}}{\log \frac1{a_b(x_j)}} u(x_j) 
&= \infty, \quad (\text{resp. } >0)\quad \text{if} \quad i_0=0\\
\label{limDiu}
\liminf_{j\to\infty} |x_j|^{n-2+i_0} a_b(x_j)^{i_0} |D^{i_0}u(x_j)| 
&= \infty,\quad (\text{resp. } >0) \quad \text{if}
\quad i_0\in \{1,2,\ldots, n-1\}.
\end{align}
Let \begin{equation}\label{rj}
r_j = |x_j|a_b(x_j). 
\end{equation}
Then by \eqref{xj}, the sequences $x_j$ and
$r_j$ satisfy \eqref{eq2.14}. Let $f_j$ be as in Lemma~\ref{lem2.3}.

Using \eqref{rj}, \eqref{limu}, and
\eqref{limDiu}, it follows from \eqref{eq2.18} and \eqref{eq2.17} with
$\xi=0$ and $|\beta|=i_0$ that
\begin{align}\label{eq3.31}
 &\liminf_{j\to\infty} \intl_{|\eta|<2} \left(\log
   \frac5{|\eta|}\right) 
f_j(\eta)\,d\eta = \infty,\quad (\text{resp. } >0)\quad \text{if}\quad i_0=0\\
\intertext{and}
\label{eq3.32}
&\liminf_{j\to\infty} \intl_{|\eta|<2} \frac{f_j(\eta)}{|\eta|^{i_0}} d\eta = \infty, \quad (\text{resp. } >0)\quad \text{if} \quad i_0 \in \{1,2,\ldots, n-1\}.
\end{align}

On the other hand, \eqref{eq2.15}, \eqref{eq3.1}, \eqref{eq2.18}, and
\eqref{eq2.17} imply for $|\xi|<1$ that
\begin{align*}
 f_j(\xi) &\le |x_j|^{2n-2} \left(\frac{r_j}{|x_j|}\right)^n 
C\sum^{n-1}_{k=0} |x_j|^{-\alpha_k}(|D^ku(x_j+r_j\xi)| + g_k(x_j+r_j\xi))^{\lambda_k}\\
&\le CB_{0j} \left(\frac{1}{\log \frac{|x_j|}{r_j}} + \vp_j + \frac1{\log \frac{|x_j|}{r_j}} \intl_{|\eta|<2}  \left(\log \frac5{|\xi-\eta|}\right) f_j(\eta)\,d\eta + G_{0j}(\xi)\right)^{\lambda_0}\\
&\quad + C\sum^{n-1}_{k=1} B_{kj} \left(\left(\frac{r_j}{|x_j|}\right)^k + \vp_j + \intl_{|\eta|<2} \frac{f_j(\eta)}{|\xi-\eta|^k} d\eta + G_{kj}(\xi)\right)^{\lambda_k}
\end{align*}
where C is a constant independent of $\xi$ and $j$  whose value may
change from line to line,
\begin{align*}
 B_{0j} &:= \frac{|x_j|^{2n-2}
   \left(\frac{r_j}{|x_j|}\right)^n|x_j|^{-\alpha_0}}{\left(\frac{|x_j|^{n-2}}{\log
       \frac{|x_j|}{r_j}}\right)^{\lambda_0}} =
 \left(\frac{a_b(x_j)\left(\log
       \frac1{a_b(x_j)}\right)^{\lambda_0/n}}{|x_j|^{b_0}}\right)^n\le C,\\
B_{kj} &:= \frac{|x_j|^{2n-2} \left(\frac{r_j}{|x_j|}\right)^n|x_j|^{-\alpha_k}}{\left(\left(\frac{r_j}{|x_j|}\right)^k |x_j|^{n-2+k}\right)^{\lambda_k}} = \left(\frac{a_b(x_j)}{|x_j|^{b_k}}\right)^{n-k\lambda_k} \le 1,\\
G_{0j}(\xi) &:= \frac{|x_j|^{n-2}}{\log \frac{|x_j|}{r_j}}
g_0(x_j+r_j\xi) =\frac{|x_j|^{n-2}}{\log \frac1{a_b(x_j)}}
g_0(x_j+r_j\xi)\le C,\\
G_{kj}(\xi) &:= \left(\frac{r_j}{|x_j|}\right)^k |x_j|^{n-2+k}
g_k(x_j+r_j\xi) = a_b(x_j)^k |x_j|^{n-2+k} g_k(x_j + r_j\xi)\le C,
\end{align*}
where we have used \eqref{loga} and \eqref{a(x)}.
Therefore for $|\xi|<1$ we have
\begin{align}\label{eq3.33}
 f_j(\xi) &\le C\bigg[\bigg(1 + \intl_{|\eta|<2} \bigg(\log \frac5{|\xi-\eta|}\bigg) f_j(\eta)\,d\eta\bigg)^{\lambda_0} + \sum^{n-1}_{k=1} \bigg(1 + \intl_{|\eta|<2} \frac{f_j(\eta)\,d\eta}{|\xi-\eta|^k} \bigg)^{\lambda_k}\bigg],\\
\label{eq3.34}
\bigg(\text{resp. } f_j(\xi) &\le C\bigg[\bigg(\vp_j + \intl_{|\eta|<2} \bigg(\log \frac5{|\xi-\eta|}\bigg) f_j(\eta)\,d\eta \bigg)^{\lambda_0} + \sum^{n-1}_{k=1} \bigg(\vp_j + \intl_{|\eta|<2} \frac{f_j(\eta)\,d\eta}{|\xi-\eta|^k} \bigg)^{\lambda_k} \bigg]\bigg),
\end{align}
where $\vp_j$ is independent of $\xi$, and $\vp_j\to 0$ as $j\to\infty$.

We can assume the $\lambda_k$ in \eqref{eq3.33} satisfy, instead of 
\eqref{l0}, the stronger condition 
\begin{equation}\label{l02}
\lambda_0>0 \quad \text{and}\quad 0<\lambda_k<n/k \quad \text{for}\quad 
k=1,2,\ldots ,n-1
\end{equation}
because slightly increasing those $\lambda_k$ in \eqref{eq3.33} which
are zero will increase the right side of \eqref{eq3.33}. By
\eqref{l0} and \eqref{lk} the $\lambda_k$ in \eqref{eq3.34}
already satisfy \eqref{l02}.

Using an argument very similar to the one used at the end of the proof
of Theorem~\ref{thm3.1} to show that \eqref{eq3.13},
(resp. \eqref{eq3.14}), leads to a contradiction of \eqref{eq3.10},
(resp. \eqref{eq3.11}), one can show that \eqref{eq3.33},
(resp. \eqref{eq3.34}), leads to a contradiction of \eqref{eq3.31},
(resp.\ \eqref{eq3.32})---the only significant difference being that
where we used Riesz potential estimates in the proof of
Theorem~\ref{thm3.1}, we must now use Riesz potential estimates
\emph{and} Lemma~\ref{lem2.5}.
\end{proof}

\begin{proof}[Proof of Theorem \ref{thm1.7}]
For  some constant $K>0$, $u$ satisfies \eqref{simple}. Thus $u$
satisfies \eqref{eq3.1} with $\lambda_k$, $\alpha_k$, and $g_k(x)$ as
in the proof of Theorem \ref{thm1.4}. Let $b$ be as in Theorem
\ref{thm3.2}. Then 
\[
b=0, \quad 
\left(\text{resp. } b=\frac{\lambda(n-2)-(2n-2)}{n}>0\right).
\]
Hence \eqref{eq1.13}, (resp. \eqref{eq1.14}), follows from
part (i), (resp. part (ii)), of Theorem \ref{thm3.2}.
\end{proof}

\begin{proof}[Proof of Theorem \ref{thm1.8}]
It follows from \eqref{eq1.15} that 
\[
 a := \frac{(n-2)(\lambda-1)-n}{\lambda} >0.
\]
Define $\psi\colon (0,1)\to (0,1)$ and $\rho\colon (0,1)\to (0,\infty)$ by 
\begin{equation}\label{eq3.44}
 \psi(r) = \max\{\sqrt{\varphi(r)}, r^{\frac{a\lambda}{2(\lambda-1)}}\}
\end{equation}
and
\begin{equation}\label{eq3.45}
 \rho(r) = \frac{n}{\lambda A} \frac{r^a}{\psi(r)^{\frac{\lambda-1}\lambda}}
\end{equation}
where $A=A(m,n)$ is as in Lemma~\ref{lem2.4}. By \eqref{eq3.44}
\begin{equation}\label{eq3.46}
 \rho(r) \le \frac{n}{\lambda A} r^{a/2}.
\end{equation}
Thus there exists a sequence $\{x_j\}^\infty_{j=1} \subset {\bb R}^n$ satisfying \eqref{eq2.25}, \eqref{eq2.26}, and
\begin{equation}\label{eq3.47}
 e^{-1} > \rho_j := \rho(|x_j|) \to 0\quad \text{as}\quad j\to\infty
\end{equation}
 such that if we define the sequence $\{r_j\}^\infty_{j=1}$ by
\begin{equation}\label{eq3.48}
 \left(\frac{|x_j|}{r_j}\right)^{n/\lambda} = \frac1{\rho_j} \log \frac1{\rho_j}
\end{equation}
then $r_j$ satisfies \eqref{eq2.27}. By \eqref{eq3.48}, \eqref{eq3.47}, and \eqref{eq3.45} we have
\begin{align}
 \log \frac{|x_j|}{r_j} &= \frac\lambda{n} \log\left(\frac1{\rho_j} \log \frac1{\rho_j}\right)\notag\\
\label{eq3.49}
&\ge \frac\lambda{n} \log \frac1{\rho_j}\\
&= \frac\lambda{n} \rho_j\left(\frac{|x_j|}{r_j}\right)^{n/\lambda}\notag\\
\label{eq3.50}
&= \frac1A \frac{|x_j|^a}{\psi(|x_j|)^{\frac{\lambda-1}\lambda}} \left(\frac{|x_j|}{r_j}\right)^{n/\lambda}.
\end{align}
Let $u$ be as in Lemma~\ref{lem2.4}. Then by \eqref{eq3.50} and Case~2 of Remark~\ref{rem2.1}, $u$ is a $C^\infty$ positive solution of \eqref{eq1.10} and by Lemma~\ref{lem2.4} we have
\[
 u(x_j) \ge \frac{A\psi(|x_j|)}{|x_j|^{n-2}} \log \frac{|x_j|}{r_j}
\]
and by \eqref{eq3.49} and \eqref{eq3.46},
\begin{align*}
 \log \frac{|x_j|}{r_j} &\ge \frac\lambda{n} \log \frac1{\rho_j} 
\ge \frac\lambda{n} \log \left(\frac{\lambda A}n |x_j|^{-a/2}\right)\\
&= \frac\lambda{n} \left(\frac{a}{2} \log \frac1{|x_j|} + \log \frac{\lambda A}n\right).
\end{align*}
Thus by \eqref{eq3.44} we have
\[
 \liminf_{j\to\infty} \frac{u(x_j)}{\sqrt{\varphi(|x_j|)}
   |x_j|^{-(n-2)} \log \frac1{|x_j|}} \ge A \frac{\lambda a}{2n} > 0
\]
from which we obtain \eqref{eq1.17}.
\end{proof}

By scaling $u$ in Theorem \ref{thm1.9}, the following theorem implies
Theorem~\ref{thm1.9}.

\begin{thm}\label{thm3.3}
Let $u(x)$ be a $C^{2m}$ nonnegative solution of
\begin{equation}\label{eq3.51}
 0 \le -\Delta^mu \le e^{u^\lambda+g^\lambda}\quad \text{in}\quad B_2(0)\backslash\{0\} \subset {\bb R}^n
\end{equation}
where $n\ge 2$ and $m\ge 2$ are integers, $m$ is odd, $2m=n$,
$0 <\lambda<1$,
and $g\colon B_2(0)\backslash\{0\} \to [0,\infty)$ is a continuous function such that
\begin{equation}\label{eq3.53}
 g(x) = o\left(|x|^{\frac{-(n-2)}{1-\lambda}}\right)\quad \text{as}\quad x\to 0.
\end{equation}
Then
\begin{equation}\label{eq3.54}
 u(x) = o\left(|x|^{\frac{-(n-2)}{1-\lambda}}\right) \quad \text{as}\quad x\to 0.
\end{equation}
\end{thm}

\begin{proof}
Suppose for contradiction that \eqref{eq3.54} does not hold. Then
there exists a sequence $\{x_j\}_{j=1}^\infty \subset {\bb R}^n$ such that
\[
 0 < 4|x_{j+1}| < |x_j| < 1/2
\]
and
\begin{equation}\label{eq3.55}
 \liminf_{j\to\infty} |x_j|^{\frac{n-2}{1-\lambda}} u(x_j)>0.
\end{equation}
Define $r_j>0$ by
\begin{equation}\label{eq3.56}
 \log \frac1{r_j} = |x_j|^{\frac{-(n-2)\lambda}{1-\lambda}}.
\end{equation}
Then
\begin{align}
 \log \frac{|x_j|}{r_j} &= \log \frac1{r_j} -  \log \frac1{|x_j|} = |x_j|^{\frac{-(n-2)\lambda}{1-\lambda}} \left[1 - |x_j|^{\frac{(n-2)\lambda}{1-\lambda}} \log \frac1{|x_j|}\right]\notag\\
\label{eq3.57}
&= |x_j|^{\frac{-(n-2)\lambda}{1-\lambda}} (1+o(1)) \quad \text{as}\quad j\to\infty.
\end{align}
So, by taking a subsequence of $j$ if necessary, we can assume $r_j<|x_j|/4$.

Let $f_j$ be as in Lemma~\ref{lem2.3}. Multiplying \eqref{eq2.18} by $|x_j|^{\frac{(n-2)\lambda}{1-\lambda}} \log \frac{|x_j|}{r_j}$ and using \eqref{eq3.57} we get for $|\xi|<1$ that
\begin{equation}\label{eq3.58}
 |x_j|^{\frac{n-2}{1-\lambda}} u(x_j+r_j\xi) \le \vp_j + |x_j|^{\frac{(n-2)}{1-\lambda}} \intl_{|\eta|<2} A\left(\log \frac5{|\xi-\eta|}\right) \frac{f_j(\eta)}{|x_j|^{n-2}}\,d\eta
\end{equation}
where the constant $A$ depends only on $m$ and $n$,
the constants $\vp_j$ are independent of $\xi$, and $\vp_j\to 0$ as $j\to\infty$. Substituting $\xi=0$ in \eqref{eq3.58} and using \eqref{eq3.55} and \eqref{eq2.16} we get
\begin{equation}\label{eq3.59}
 \liminf_{j\to\infty} |x_j|^{\frac{n-2}{1-\lambda}} \intl_{|\eta|<1} \left(\log \frac5{|\eta|}\right) \frac{f_j(\eta)}{|x_j|^{n-2}}\,d\eta > 0.
\end{equation}
By \eqref{eq2.15}, \eqref{eq3.51}, \eqref{eq3.58}, and \eqref{eq3.53} we have  
\begin{equation}\label{eq3.60}
 \frac{f_j(\xi)}{|x_j|^{n-2}r^n_j} \le
 e^{u_j(\xi)^\lambda+M^\lambda_j} \quad\text{for} \quad |\xi|<1
\end{equation}
where
\[
 u_j(\xi) = \intl_{|\eta|<2} A\left(\log \frac5{|\xi-\eta|}\right) \frac{f_j(\eta)}{|x_j|^{n-2}}\,d\eta
\]
and the positive constants $M_j$ satisfy
\begin{equation}\label{eq3.61}
 M_j|x_j|^{\frac{n-2}{1-\lambda}} \to 0 \quad \text{and}\quad
 M_j\to\infty \quad \text{as} \quad j\to\infty.
\end{equation}
Let $\Omega_j =\{\xi\in B_1(0)\colon u_j(\xi) > M_j\}$. Then for $\xi\in \Omega_j$ it follows from \eqref{eq3.60} that
\begin{align}
 \frac{f_j(\xi)^2}{(|x_j|^{n-2}r^n_j)^2} &\le e^{4u_j(\xi)^\lambda}\notag\\
\label{eq3.62}
&\le \exp\left[\left(\, \intl_{|\eta|<2} b_j\left(\log \frac5{|\xi-\eta|}\right) \frac{f_j(\eta)}{\int_{B_2}f_j} d\eta\right)^\lambda\right]
\end{align}
where
\[
 b_j = 4^{1/\lambda} A|x_j|^{-(n-2)} \max\left\{\intl_{B_2} f_j, |x_j|^{\frac{n-2}2}\right\}.
\]
By \eqref{eq2.16},
\begin{equation}\label{eq3.63}
 b_j|x_j|^{n-2}\to 0\quad \text{and}\quad
 b_j\to\infty\quad\text{as}\quad j\to\infty.
\end{equation}
Hence by \eqref{eq3.62}, Jensen's inequality, and the fact that $\exp(t^\lambda)$ is concave up for $t$ large we have for $\xi\in\Omega_j$ that
\[
 \frac{f_j(\xi)^2}{(|x_j|^{n-2}r^n_j)^2} \le \intl_{|\eta|<2} \exp\left(b^\lambda_j \left(\log \frac5{|\xi-\eta|}\right)^\lambda\right) \frac{f_j(\eta)}{\int_{B_2}f_j}\, d\eta
\]
and consequently
\begin{align}
\intl_{\Omega_j} \frac{f_j(\xi)^2}{(|x_j|^{n-2}r^n_j)^2}\,d\xi &\le
\intl_{|\eta|<2}\left(\, \intl_{|\xi|<1} \exp\left(b^\lambda_j\left(\log \frac5{|\xi-\eta|}\right)^\lambda\right) d\xi\right) \frac{f_j(\eta)}{\int_{B_2} f_j}\,d\eta\notag\\
&\le \max_{|\eta|\le 2} \intl_{|\xi|<1} \exp\left(b^\lambda_j\left(\log \frac5{|\xi-\eta|}\right)^\lambda\right)\, d\xi\notag\\
&= \intl_{|\xi|<1} \exp\left(b^\lambda_j\left(\log \frac5{|\xi|}\right)^\lambda\right) \, d\xi\notag\\ 
\label{eq3.64}
&= I_1+I_2
\end{align}
where
\[
 I_1 = \intl_{\underset{\sst |\xi|<1}{\log \frac5{|\xi|}<(b^\lambda_j\lambda)^{\frac1{1-\lambda}}}} \exp\left(b^\lambda_j(\log \frac5{|\xi|})^\lambda\right)\,d\xi\quad \text{and}\quad I_2 = \intl_{\log \frac5{|\xi|}>(b^\lambda_j\lambda)^{\frac1{1-\lambda}}} \exp\left(b^\lambda_j(\log \frac5{|\xi|})^\lambda\right) \,d\xi.
\]
Clearly
\[
 \frac{I_1}{|B_1(0)|} \le \exp\left((b_j(b^\lambda_j\lambda)^{\frac1{1-\lambda}})^\lambda\right) = \exp\left((b_j\lambda)^{\frac\lambda{1-\lambda}}\right) \le \exp\left(b^{\frac\lambda{1-\lambda}}_j\right)
\]
and using Jensen's inequality and the fact that $e^{b^\lambda_j(\log t)^\lambda}$ is concave down as a function of $t$ for $\log t > (b^\lambda_j\lambda)^{\frac1{1-\lambda}}$ one can show that
\[
 I_2 \le \exp\left(C b^{\frac\lambda{1-\lambda}}_j\right)
\]
where $C$ depends only on $n$. Therefore by \eqref{eq3.64} and \eqref{eq3.56},
\begin{align*}
 \intl_{\Omega_j} \left(\frac{f_j(\xi)}{|x_j|^{n-2}}\right)^2 d\xi &\le  \exp\left(-2n |x_j|^{\frac{-(n-2)\lambda}{1-\lambda}}\right) \exp\left(Cb^{\frac\lambda{1-\lambda}}_j\right)\\
&= \exp \left(C b^{\frac\lambda{1-\lambda}}_j - 2n|x_j|^{\frac{-(n-2)\lambda}{1-\lambda}}\right) \to 0\quad \text{as}\quad j\to\infty
\end{align*}
by \eqref{eq3.63}. Thus by H\"older's inequality
\[
 \lim_{j\to \infty} \intl_{\Omega_j} \left(\log \frac5{|\eta|}\right) \frac{f_j(\eta)}{|x_j|^{n-2}}\,d\eta = 0.
\]
Hence defining $g_j\colon  B_1(0)\to [0,\infty)$ by
\[
 g_j(\xi) := \begin{cases}
              f_j(\xi)&\text{for $\xi\in B_1\backslash\Omega_j$,}\\
0&\text{for $\xi\in \Omega_j$,}
             \end{cases}
\]
it follows from \eqref{eq3.59} and \eqref{eq3.61} that
\begin{equation}\label{eq3.65}
 \frac1{M^\lambda_j} \intl_{|\eta|<1} \left(\log \frac5{|\eta|}\right) g_j(\eta)\,d\eta \to \infty \quad \text{as}\quad j\to \infty.
\end{equation}
By \eqref{eq2.16}, we have
\begin{equation}\label{eq3.66}
 \intl_{|\eta|<1} g_j(\eta)\,d\eta\to 0\quad \text{as}\quad j\to\infty
\end{equation}
and by \eqref{eq3.60} we have
\begin{equation}\label{eq3.67}
 g_j(\xi) \le e^{2M^\lambda_j}\quad \text{in}\quad B_1(0).
\end{equation}

For fixed $j$, think of $g_j(\eta)$ as the density of a distribution of mass in $B_1$ satisfying \eqref{eq3.65}, \eqref{eq3.66}, and \eqref{eq3.67}. By moving small pieces of this mass nearer to the origin in such a way that the new density (which we again denote by $g_j(\eta))$ does not violate \eqref{eq3.67}, we will not change the total mass $\int_{B_1} g_j(\eta)\,d\eta$ but $\int_{B_1}(\log 5/|\eta|) g_j(\eta)\,d\eta$ will increase. Thus for some $\rho_j\in (0,1)$ the functions
\[
 g_j(\eta) = \begin{cases}
              e^{2M^\lambda_j}&\text{for $|\eta|<\rho_j$,}\\
0&\text{for $\rho_j<|\eta|<1$}
             \end{cases}
\]
satisfy \eqref{eq3.65}, \eqref{eq3.66}, and \eqref{eq3.67}, which, as elementary and explicit calculations show, is impossible because $M_j\to\infty$ as $j\to \infty$. This contradiction proves Theorem~\ref{thm3.3}.
\end{proof}

\begin{proof}[Proof of Theorem \ref{thm1.10}]
Define $\psi\colon (0,1)\to (0,1)$ by 
\[
 \psi(r) = \max\left\{\varphi(r)^{\frac{1-\lambda}2}, r^{\frac{n-2}2}\right\}.
\]
Since $\psi(r)\ge r^{\frac{n-2}2}$ there exists a sequence
$\{x_j\}^\infty_{j=1} \subset {\bb R}^n$ satisfying \eqref{eq2.25} and
\eqref{eq2.26} such that if we define the sequence
$\{r_j\}^\infty_{j=1}\subset (0,\infty)$ by 
\begin{equation}\label{eq3.68}
 \log \frac{|x_j|}{r_j} = \left[\frac1{2n} \left(\frac{A\psi(|x_j|)}{|x_j|^{n-2}}\right)^\lambda \right]^{\frac1{1-\lambda}},
\end{equation}
where $A=A(m,n)$ is as in Lemma \ref{lem2.4}, then $r_j$ will satisfy \eqref{eq2.27} and
\[
 \log \frac1{|x_j|^{2n-2}} < \log \frac{|x_j|}{r_j}.
\]
Thus
\begin{align}
 \frac{\log \frac{\psi(|x_j|)}{|x_j|^{2n-2}} + n \log \frac{|x_j|}{r_j}}{\left(\frac{A \psi(|x_j|)}{|x_j|^{n-2}} \log \frac{|x_j|}{r_j}\right)^\lambda} &\le \frac{2n \log \frac{|x_j|}{r_j}}{\left(\frac{A\psi(|x_j|)}{|x_j|^{n-2}} \log \frac{|x_j|}{r_j}\right)^\lambda}\notag\\
\label{eq3.69}
&= \frac{\left(\log \frac{|x_j|}{r_j}\right)^{1-\lambda}}{\frac1{2n} \left(\frac{A \psi(|x_j|)}{|x_j|^{n-2}}\right)^\lambda} = 1.
\end{align}
Let $u$ be as in Lemma~\ref{lem2.4}. Then by \eqref{eq3.69} and Case~2 of Remark~\ref{rem2.1}, $u$ is a $C^\infty$ positive solution of \eqref{eq1.20} and by Lemma~\ref{lem2.4} and \eqref{eq3.68} we have
\begin{align*}
 u(x_j) &\ge \frac{A\psi(|x_j|)}{|x_j|^{n-2}} \left[\frac1{2n} \left(\frac{A\psi(|x_j|)}{|x_j|^{n-2}}\right)^\lambda\right]^{\frac1{1-\lambda}}\\
&= \left(\frac{A\psi(|x_j|)}{2n}\right)^{\frac1{1-\lambda}} \frac1{|x_j|^{\frac{n-2}{1-\lambda}}}\\
&\ge \left(\frac{A}{2n}\right)^{\frac1{1-\lambda}} \sqrt{\varphi(|x_j|)}\ |x_j|^{-\frac{n-2}{1-\lambda}}
\end{align*}
which implies \eqref{eq1.21}.
\end{proof}

\begin{proof}[Proof of Theorem \ref{thm1.11}]
Define $\psi\colon (0,1)\to (0,1)$ by $\psi(r) = r^{\frac{n-2}2}$. Choose a sequence $\{x_j\}^\infty_{j=1} \subset {\bb R}^n$ satisfying \eqref{eq2.25}, \eqref{eq2.26}, and
\begin{equation}\label{eq3.70}
 \frac{A\psi(|x_j|)}{|x_j|^{n-2}} > n+1
\end{equation}
where $A=A(m,n)$ is as in Lemma~\ref{lem2.4}. Choose a sequence $\{r_j\}^\infty_{j=1} \subset {\bb R}$ satisfying \eqref{eq2.27},
\begin{equation}\label{eq3.71}
 \log \frac1{|x_j|^{2n-2}} < \log \frac{|x_j|}{r_j}
\end{equation}
and
\begin{equation}\label{eq3.72}
 (n+1) \log \frac{|x_j|}{r_j} > \varphi(|x_j|)^2.
\end{equation}
Then, since $\psi(|x_j|)<1$, it follows from 
\eqref{eq3.71} and \eqref{eq3.70} that
\begin{equation}\label{eq3.73}
 \log \frac{\psi(|x_j|)}{|x_j|^{2n-2}} + n \log \frac{|x_j|}{r_j} \le (n+1) \log \frac{|x_j|}{r_j} \le \left((n+1) \log \frac{|x_j|}{r_j}\right)^\lambda \le \left(\frac{A\psi(|x_j|)}{|x_j|^{n-2}} \log \frac{|x_j|}{r_j}\right)^\lambda.
\end{equation}
Let $u$ be as in Lemma~\ref{lem2.4}. Then by \eqref{eq3.73} and Case~2 of Remark~\ref{rem2.1}, $u$ is a $C^\infty$ positive solution of \eqref{eq1.20} and by Lemma~\ref{lem2.4}, \eqref{eq3.70} and \eqref{eq3.72} we have
\[
 u(x_j) \ge \varphi(x_j)^2
\]
which implies \eqref{eq1.23}.
\end{proof}

We now prove Proposition \ref{pro1.1}. It follows immediately from the
following more general proposition, which is easier to prove than
Proposition \ref{pro1.1}.

\begin{pro}\label{pro3.1}
Suppose $u$ is a $C^{2m}$ radial solution of
\begin{equation}\label{eq3.75}
\Delta^mu\le A\Gamma \quad \text{in}\quad \overline{B_1(0)}\backslash
\{0\}\subset {\bb R}^n,
\end{equation}
where $m\ge 1$ and $n\ge 2$ are integers, $\Gamma$ is given by
\eqref{eq1.3}, and $A$ is a positive constant. Then
\begin{equation}\label{eq3.76}
u\le C\Gamma\quad\text{in}\quad\overline{B_1(0)}\backslash\{0\}
\end{equation}
for some positive constant $C$. 
\end{pro}

\begin{proof} 
We use induction. Suppose $m=1$. Let $g=\Delta u$. Then for $0<r\le 1$
we have 
\begin{equation}\label{eq3.77}
u(r)=u(1)-u'(1)\int_r^1 \rho^{1-n}\,d\rho + \int_r^1 s^{1-n}
\int_s^1 g(\rho)\rho^{n-1}\,d\rho\,ds.
\end{equation}
Since $g\le A\Gamma$ we have 
\[
\int_s^1g(\rho)\rho^{n-1}\, d\rho \le
A\int_0^1\Gamma(\rho)\rho^{n-1}\,d\rho <\infty.
\]
Hence \eqref{eq3.76} follows from \eqref{eq3.77}.

Suppose inductively that the propositon is true for $m-1$ where $m\ge
2$ and $u$ is a $C^{2m}$ radial solution of \eqref{eq3.75}. Then by
the $m=1$ case, $\Delta^{m-1}u\le A_1\Gamma$ in 
$\overline{B_1(0)}\backslash\{0\}$ for some positive constant
$A_1$. Thus \eqref{eq3.76} follows from the inductive assumption.
\end{proof}

\section{Proofs when the singularity is at infinity}\label{sec4}
 
In this section we prove Theorems~\ref{thm1.14}--\ref{thm1.17}, which
deal with the case that the singularity is at infinity.

Since the proofs of Theorems \ref{thm1.14} and \ref{thm1.15} are
similar, we prove them together.

\begin{proof}[Proof of Theorem \ref{thm1.14}, (resp. \ref{thm1.15})]
For some constant $K>0$, $v(y)$ satisfies
\[
0\le -\Delta^mv\le K(v+1)^\lambda \quad\text{in}\quad {\bb
  R}^n\backslash B_{1/2}(0).
\]
Let $u(x)$ be defined by \eqref{eq1.26}. Then, by \eqref{eq1.27}, $u(x)$
is a nonnegative solution of 
\[
0\le -\Delta^mu\le K|x|^{-\alpha_0}(u+g_0(x))^\lambda \quad\text{in}\quad 
  B_{2}(0)\backslash \{0\},
\]
where $\alpha_0=2m+n-(n-2m)\lambda$ and 
\[
g_0(x)=|x|^{-(n-2m)}=o(|x|^{-(n-2)}) \quad\text{as } x\to 0.
\]
Thus $u$ satisfies \eqref{eq3.1} and \eqref{eqnew1.1} with
$\lambda_0=\lambda$ and $\lambda_k=1$, $\alpha_k=-(n-2+k)$,
$g_k(x)\equiv 1$ for $k=1,2,\ldots,2m-1$.
Using these values of $\lambda_k$ and $\alpha_k$ in $b$, as defined in
Theorem \ref{thm3.1}, (resp. \ref{thm3.2}), we get 
\[
b=\frac{\lambda(2m-2)+2}{n-\lambda(n-2m)}>0.
\]
It therefore follows from part (ii) of Theorem \ref{thm3.1},
(resp. \ref{thm3.2}), that as $x\to 0$ we have
\[
u(x)=o(|x|^{-a_0}), \quad\left(\text{resp.}\quad
u(x)=o(|x|^{-(n-2)}\log\frac{5}{|x|})\right), 
\]
where $a_0=(n-2m)b+n-2$. This estimate for $u(x)$ is equivalent, via
\eqref{eq1.26}, to \eqref{eq1.29}, (resp. \eqref{new}).
\end{proof}

By scaling and translating $v$ in Theorem \ref{thm1.16}, we see that
the following theorem implies Theorem~\ref{thm1.16}.

\begin{thm}\label{thm4.3}
Let $v(y)$ be a $C^{2m}$ nonnegative solution of
\begin{equation}\label{eq4.12}
0 \le -\Delta^m v \le e^{v^\lambda+g^\lambda}\quad \text{in}\quad {\bb R}^n\backslash B_{1/2}(0)
\end{equation}
where $m\ge 2$ and $n\ge 2$ are integers, $m$ is odd, $2m=n$, $0<\lambda<1$, and $g\colon {\bb R}^n \backslash B_{1/2}(0) \to [1,\infty)$ is a continuous function satisfying
\begin{align}\label{eq4.13}
g(y) &= o(|y|^{\frac{n-2}{1-\lambda}})\quad \text{as}\quad |y|\to\infty.\\
\intertext{Then}
\label{eq4.14}
v(y) &= o(|y|^{\frac{n-2}{1-\lambda}})\quad \text{as}\quad |y|\to\infty.
\end{align}
\end{thm}

\begin{proof}
Let $u(x)$ be defined by \eqref{eq1.26}. Then by \eqref{eq4.12} and \eqref{eq1.27} we have
\begin{align*}
0 &\le -|x|^{2n} \Delta^mu(x) \le \exp\left(u(x)^\lambda + g\left(\frac{x}{|x|^2}\right)^\lambda\right)
\intertext{and thus by \eqref{eq4.13},}
0 &\le -\Delta^mu(x) \le \exp\left(u(x)^\lambda +
  o\left(\left(\frac1{|x|}\right)\right)^{\frac{\lambda(n-2)}{1-\lambda}}\right)
\quad\text{in}\quad B_2(0)\backslash\{0\}.
\end{align*}
Hence Theorem \ref{thm3.3} implies
\[
u(x) = o(|x|^{-\frac{(n-2)}{1-\lambda}})\quad \text{as}\quad x\to 0
\]
and so \eqref{eq4.14} holds.
\end{proof}

\begin{proof}[Proof of Theorem \ref{thm1.17}]
By using the $m$-Kelvin transform \eqref{eq1.26}, we see that to prove Theorem~\ref{thm1.17} it suffices to prove that there exists a $C^\infty$ positive solution $u(x)$ of
\begin{equation}\label{eq4.15}
0 \le -\Delta^m u \le |x|^\tau u^\lambda\quad \text{in}\quad {\bb R}^n\backslash\{0\},
\end{equation}
where
\[
\tau = \lambda(n-2m) - n - 2m
\]
such that
\begin{equation}\label{eq4.16}
u(x) \ne O(\varphi(|x|^{-1}) |x|^{-(a+n-2m)})\quad \text{as}\quad x\to 0.
\end{equation}
Define $\psi\colon (0,1)\to (0,1)$ by
\begin{equation}\label{eq4.17}
\psi(r) = \max\left\{\varphi(r^{-1})^p, r^{b\frac{n-\lambda(n-2m)}\lambda}\right\}
\end{equation}
where
\[
b := \frac{\lambda(m-1)+1}{n-\lambda(n-2m)} \quad \text{and}\quad p := \frac{n-\lambda(n-2m)}{2n}.
\]
By \eqref{eq1.30}, $b$ and $p$ are positive. Also 
\begin{equation}\label{eq4.18}
1 + 2b = \frac{\lambda(2m-2)-2m+2-\tau}{n-\lambda(n-2m)}\quad \text{and}\quad a = 2m-2+(n-2m)2b.
\end{equation}
Let $\{x_j\}^\infty_{j=1} \subset {\bb R}^n$ be a sequence satisfying \eqref{eq2.25} and \eqref{eq2.26}. Define $r_j>0$ by
\[
r^{n-\lambda(n-2m)}_j = \frac{2^{|\tau|}}{A^\lambda} \frac{|x_j|^{\lambda(2m-2)-2m+2-\tau}}{\psi(|x_j|)^\lambda}
\]
where $A=A(m,n)$ is as in Lemma~\ref{lem2.4}. Then $r_j$ satisfies
\eqref{eq2.36-2} and by \eqref{eq4.17} and \eqref{eq4.18},
\begin{align}\label{eq4.19}
r_j &= C(m,n,\lambda) \frac{|x_j|^{1+2b}}{\psi(|x_j|)^{\frac\lambda{n-\lambda(n-2m)}}}\\
&\le C(m,n,\lambda) |x_j|^{1+b}.\notag
\end{align}
Thus by taking a subsequence of $j,r_j$ will satisfy \eqref{eq2.27}. Let $u$ be as  in Lemma~\ref{lem2.4}. Then by Case~I of Remark~\ref{rem2.1}, $u$ is a $C^\infty$ positive solution of \eqref{eq4.15} and by \eqref{eq2.30}, \eqref{eq4.17}, \eqref{eq4.18}, and \eqref{eq4.19} we have  
\begin{align*}
u(x_j) &\ge \frac{C(m,n,\lambda)\psi(|x_j|)}{|x_j|^{2m-2}} \frac{\psi(|x_j|)^{\frac{\lambda(n-2m)}{n-\lambda(n-2m)}}}{|x_j|^{(1+2b)(n-2m)}}\\
&= \frac{C(m,n,\lambda) \psi(|x_j|)^{\frac{n}{n-\lambda(n-2m)}}}{|x_j|^{(n-2m) + (2m-2) + (n-2m)2b}}\\
&\ge C(m,n,\lambda) \frac{\varphi(|x_j|^{-1})^{1/2}}{|x_j|^{a+n-2m}}
\end{align*}
which implies \eqref{eq4.16}.
\end{proof}

\end{document}